\documentclass[11pt]{amsart}
\usepackage{amsmath,amssymb,amsfonts,latexsym}%,mathrsfs}
\usepackage[latin1]{inputenc}
\usepackage{amsthm}
\usepackage{verbatim}
\usepackage{array}
\usepackage{graphicx,color}
\usepackage{graphics}
\usepackage{epsfig}

\theoremstyle{plain}
\newtheorem{theorem}{Theorem}%[section]
\newtheorem{lemma}[theorem]{Lemma}
\newtheorem{proposition}[theorem]{Proposition}
\newtheorem{corollary}[theorem]{Corollary}

\theoremstyle{definition}

\newtheorem{remark}[theorem]{Remark}

\newcommand{\bo}{{\rm O}}
\newcommand{\ds}{\displaystyle}
\newcommand{\so}{{\rm o}}

\newcommand{\eqskip}{ \vspace*{2mm}\\ }

\title[Optimal eigenvalues of the Robin Laplacian]{Asymptotic behaviour and
numerical approximation of optimal eigenvalues of the Robin Laplacian
\thanks{P.R.S.A. was supported by Funda\c c\~{a}o para a Ci\^{e}ncia e Tecnologia (FCT), Portugal, through grant SFRH/BPD/47595/2008 and project PTDC/MAT/105475/2008 and by Funda\c{c}\~{a}o Calouste Gulbenkian through
program \emph{Est\'{\i}mulo \`{a} Investiga\c c\~{a}o 2009}. J.B.K. was supported by a grant within the scope
of FCT's project PTDC/\ MAT/\ 101007/2008. All authors were partially supported by FCT's projects PTDC/\ MAT/\ 101007/2008 and PEst-OE/MAT/UI0208/2011}}
\author[P.R.S. Antunes, P. Freitas and J.B. Kennedy]{Pedro R.S. Antunes$^{1,3}$ \and Pedro Freitas$^{1,2}$ \and James B. Kennedy$^1$
}

\date{\today}

\keywords{Robin Laplacian, Eigenvalues, Optimisation}

\subjclass[2000]{35P15 \and 35J05 \and 49Q10 \and  65N25}

\begin{document}

\maketitle

\begin{abstract}
We consider the problem of minimising the $n^{th}-$eigenvalue of the Robin Laplacian in $\mathbb{R}^{N}$. Although
for $n=1,2$ and a positive boundary parameter $\alpha$ it is known that the minimisers do not depend on $\alpha$, we
demonstrate numerically that this will not always be the case and illustrate how the optimiser will depend on
$\alpha$. We derive a Wolf-Keller type result for this problem and show that optimal eigenvalues grow at most with
$n^{1/N}$, which is in sharp contrast with the Weyl asymptotics for a fixed domain. We further show that the gap
between consecutive eigenvalues does go to zero as $n$ goes to infinity. Numerical results then support the
conjecture that for each $n$ there exists a positive value of $\alpha_{n}$ such that the $n^{\rm th}$ eigenvalue is
minimised by $n$ disks for all $0<\alpha<\alpha_{n}$ and, combined with analytic estimates, that this value is
expected to grow with $n^{1/N}$.
\end{abstract}

\vspace*{2cm}

($^1$) Group of Mathematical Physics of the University of Lisbon, Complexo Interdisciplinar, Av.~Prof.~Gama Pinto~2, P-1649-003 Lisboa, Portugal, Tel.: +351--217904857, Fax: +351--217954288

{\tt e-mail address}: pant@cii.fc.ul.pt, jkennedy@cii.fc.ul.pt\vspace*{5mm}

($^2$) Department of Mathematics, Faculty of Human Kinetics of the Technical University of Lisbon {\rm and} Group of Mathematical Physics of the University of Lisbon, Complexo Interdisciplinar, Av.~Prof.~Gama Pinto~2, P-1649-003 Lisboa, Portugal, Tel.: +351--217904852, Fax: +351--217954288

{\tt e-mail address}: freitas@cii.fc.ul.pt\vspace*{5mm}

($^3$) Department of Mathematics, Universidade Lus\'{o}fona de Humanidades e Tecnologias, Av. do Campo Grande, 376, 1749-024 Lisboa, Portugal

\section{Introduction}

Optimisation of eigenvalues of the Laplace operator is a classic topic in spectral theory, going back to the work of Rayleigh at the end of the nineteenth century. The first result of this type is the well-known Rayleigh--Faber--Krahn inequality which states that among all Euclidean domains of fixed volume the ball minimises the first Dirichlet eigenvalue~\cite{rayl,fab,krahn1,krahn2}. As a more or less direct consequence of this result, it is possible to obtain that the second Dirichlet eigenvalue is minimised by two balls of equal volume. The case of other boundary conditions has also received some attention and it has been known since the $1950$'s that the ball is a maximiser for the first nontrivial Neumann eigenvalue~\cite{sz,wein} and, more recently, that it also minimises the first Robin eigenvalue with a positive boundary parameter~\cite{boss,buda,dane}, while two equal balls are again the minimiser for the second eigenvalue~\cite{kenn1}.

In spite of this, for higher eigenvalues very little is known and even proving existence of
minimisers in the Dirichlet case poses great difficulties. Bucur and Henrot showed the existence
of a minimiser for the third eigenvalue among quasi-open sets in $2000$~\cite{buhe}, and it is only
very recently that Bucur~\cite{bucu} and Mazzoleni and Pratelli~\cite{mapr} proved, independently, the existence of bounded minimisers for all eigenvalues in the context of quasi-open sets.

Moreover, in the planar Dirichlet case and apart from the third and fourth eigenvalues, minimisers,
are expected to be neither balls nor unions of balls, as is already known up to the fifteenth eigenvalue~\cite{anfr}. In fact, it is not to be expected that the boundaries of optimisers can explicitly be described in terms of known functions either, which means that the type of result that one may look for should be of a different nature from the Rayleigh--Faber--Krahn type.

For instance, and always assuming existence of minimisers, there are several qualitative questions which may be raised with respect to this and related problems and which include multiplicity issues, symmetry and connectedness properties, to name just a few. However, even such results might not hold in full generality, as some recent numerical results seem to indicate~\cite{anfr}.

All of the above issues make this field of research suitable ground for the combination of rigorous analytic methods with accurate numerical calculations in order to explore the properties of such problems. Indeed, and although numerical analysis of eigenvalue problems goes back many years, within the last decade there have been several extensive numerical studies based on new methods which allow us to obtain insight into the behaviour of such problems. To mention just two of the most recent related to eigenvalue optimisation, see~\cite{oude,anfr} for the optimisation of Dirichlet and Dirichlet and Neumann eigenvalues, respectively.

The purpose of this paper is to consider the optimisation of higher eigenvalues $\lambda_{n}$ of the $N$-dimensional Robin eigenvalue problem and analyse some of its properties, combining both the approaches mentioned above. From a theoretical perspective, we begin by establishing a Wolf--Keller type result, which is needed in the numerical optimisation procedure in order to check for non--connected optimal domains. We then consider the asymptotic behaviour of both optimal values of $\lambda_{n}$, which we shall call $\lambda_{n}^{*}$, and the difference between $\lambda_{n+1}^{*}$ and $\lambda_{n}^{*}$. The main result here is the fact that $\lambda_{n}^{*}$ grows at most with $n^{1/N}$ as $n$ goes to infinity, and that the difference between optima does go to zero in this limit.
Note that this asymptotic behaviour for optimal eigenvalues is in sharp contrast with Weyl's law for the behaviour
of the high frequencies for a fixed domain $\Omega$, namely,
\[
\lambda_{n}(\Omega) = \frac{\ds 4\pi^{2}}{\ds \left(\omega_{N} |\Omega|\right)^{2/N}}n^{2/N} + \so(n^{2/N})
\;\;\mbox{ as } n \to \infty,
\]
where $\omega_N$ denotes the volume of the ball of unit radius in $\mathbb{R}^N$ and
$|\Omega|$ is the $N$-dimensional volume of $\Omega$.
Finally, we prove some results regarding the behaviour of $\lambda_{n}(t\Omega,\alpha)$ as a function of the parameters $t$ and $\alpha$. Although intuitively obvious and part of the folklore, their proofs do not seem entirely trivial and it is difficult to source them precisely in the literature. Hence we have included proofs.

At the numerical level, our results are obtained using a meshless method known as the Method of Fundamental Solutions. Since it is, as far as we know, the first time that such a method has been applied to the Robin problem, we begin by describing it and stating some basic properties. We then present the results of the optimisation procedure. This allows us to conclude numerically, as was observed in \cite[Sec.~3]{kenn2}, that the optimiser will depend on the value of $\alpha$ for $n$ larger than two, and provides support for the conjecture that for small positive values of $\alpha$ the $n^{\rm th}$ eigenvalue is minimised by $n$ identical balls. In fact, and assuming that the domain comprising $n$ equal balls stops being a minimiser when its $n^{\rm th}$ eigenvalue becomes larger than that of the set formed by $n-3$ small balls and a larger ball, we show that the value of $\alpha$ at which this happens is increasing with $n$ and grows to infinity.

The paper is divided into three parts. In the first we present the analytic results described above, together with the corresponding proofs. This is followed by a description of the numerical method used, and finally we present the numerical results obtained.

\section{Theoretical results}

We write the eigenvalue problem as
\begin{equation}
    \label{eq:robin}
    \begin{aligned}
        -\Delta u&= \lambda u &\quad &\text{in $\Omega$},\\
        \frac{\partial u}{\partial\nu}+\alpha u&=0&&
        \text{on $\partial\Omega$}
    \end{aligned}
\end{equation}
where $\nu$ is the outer unit normal to $\Omega$ and the boundary parameter $\alpha>0$ is a constant. We will assume throughout this section that $\Omega \subset \mathbb{R}^N$ is a bounded, open set with Lipschitz boundary, not necessarily connected, with $N$-dimensional volume $|\Omega|$ equal to some fixed constant $V>0$. We will also use $\sigma$ to denote surface measure. As is standard, we will always interpret the problem \eqref{eq:robin} in the weak sense, so that an eigenvalue $\lambda\in\mathbb{R}$ and associated eigenfunction $u \in H^1(\Omega)$ solve the equation
\begin{equation}
	\label{eq:weak}
	\int_\Omega \nabla u\cdot\nabla v\,dx + \int_{\partial\Omega} \alpha uv\,d\sigma = \lambda\int_\Omega uv\,dx
\end{equation}
for all $v\in H^1(\Omega)$. It is well known that for each $\Omega \subset \mathbb{R}^N$ and $\alpha>0$, there is a discrete set of eigenvalues $\{\lambda_n(\Omega,\alpha)\}_{n\geq 1}$, all positive, ordered by increasing size, and repeated according to their respective multiplicities. For each $n\geq 1$, we are interested in the quantity
\begin{equation}
	\label{eq:ninf}
	\lambda_n^*=\lambda_n^*(V,\alpha):=\inf\{\lambda_n(\Omega,\alpha):|\Omega|=V\}
\end{equation}
for each fixed $V>0$ and $\alpha\geq 0$, where we assume $\Omega$ belongs to the class of all bounded, open, Lipschitz subsets of $\mathbb{R}^N$, as well as the properties of any associated minimising domain(s) $\Omega^* = \Omega^* (n,V,\alpha,N)$. As for the Dirichlet problem, when $n=1$, the unique minimising domain is a ball \cite{buda,dane}, while for $n=2$ it is the union of two equal balls \cite{kenn1,kenn2}. Unlike in the Dirichlet case, no existence result is known for any $n\geq 3$; in $\mathbb{R}^2$, it was shown in \cite{kenn2} that, for each $n\geq 3$, there cannot be a minimiser independent of $\alpha>0$. As the dependence of $\lambda_n^*(V,\alpha)$ on $\alpha \geq 0$ is one of the principal themes of this paper, we note here the following basic properties of this function. The proof will be deferred until Sec.~\ref{sec:append}.

\begin{proposition}
\label{prop:starvsalpha}
Let $V>0$ and $n \geq 1$ be fixed and for each $\alpha\geq 0$ let $\lambda_n^*(V,\alpha)$ be given by \eqref{eq:ninf}. Then as a function of $\alpha \in [0,\infty)$, $\lambda_n^*(V,\alpha)$ is continuous and strictly monotonically increasing, with $\lambda_n^*(V,0) = 0$ and $\lambda_n^*(V,\alpha) < \lambda_n^*(V,\infty)$, the infimal value for the corresponding Dirichlet problem.
\end{proposition}

\begin{remark}
\label{rem:existence}
Throughout this section, we will tend to assume for simplicity, especially in the proofs, that \eqref{eq:ninf} does in fact possess a minimiser $\Omega^*$, for each $n\geq 1$ and $\alpha>0$. This assumption can easily be removed by considering an arbitrary sequence of domains $\Omega_k^*$ with $\lambda_n(\Omega_k^*) \to \lambda_n^*$. As this type of argument is quite standard, we omit the details. Note that without loss of generality each domain $\Omega_k^*$ may be assumed to have at most $n$ connected components, as more could only increase $\lambda_n$ (see \cite[Remark~3.2(ii)]{kenn2}). In fact, such a sequence may be assumed to be connected. This is a consequence of results on the stability of solutions to \eqref{eq:robin} with respect to domain perturbation: for any domain $\Omega$ with $n$ connected components, any $\alpha>0$ and any $\varepsilon>0$, by \cite[Corollary~3.7]{danc}, there exists another ``dumbbell"-type connected domain $\Omega'$, which has narrow passages joining the disconnected components of $\Omega$, such that $|\Omega'|=|\Omega|$ and $\lambda_n(\Omega',\alpha) \leq \lambda_n(\Omega,\alpha)+\varepsilon$ (cf.~also \cite[Example~2.2]{kenn1}).
\end{remark}

Our point of departure is the way the Robin problem behaves under homothetic scaling of the domain. That is, if we denote by $t\Omega$ the rescaled domain $\{tx \in \mathbb{R}^N: x \in \Omega\}$, then, by a simple change of variables, \eqref{eq:robin} is equivalent to
\begin{equation}
    \label{eq:scaled}
    \begin{aligned}
        -\Delta u&= \frac{\lambda}{t^2} u &\quad &\text{in $t\Omega$},\\
        \frac{\partial u}{\partial\nu}+\frac{\alpha}{t} u&=0&&
        \text{on $\partial (t\Omega)$},
    \end{aligned}
\end{equation}
that is, $\lambda_n(\Omega,\alpha)=t^2\lambda_n(t\Omega,\alpha/t)$, for all (bounded, Lipschitz) $\Omega\subset\mathbb{R}^N$, $n \geq 1$, $\alpha > 0$ and $t>0$. (This of course remains valid considering the weak form \eqref{eq:weak}.) We highlight the change in the boundary parameter. This means that, unlike in the Dirichlet and Neumann cases, $|t\Omega|^\frac{2}{N} \lambda_n(t\Omega,\alpha)$ is \emph{not} invariant with respect to changes in $t>0$; rather, the invariant quantity is
\begin{equation}
	\label{eq:invariant}
	|t\Omega|^{2/N} \lambda_n(t\Omega,\alpha/t).
\end{equation}
This will have a profound effect on the nature of the minimising value $\lambda_n^*$ and any corresponding minimising domains. We observe that $\lambda_n^*(V,\alpha)$ as a function of $\alpha\in (0,\infty)$ may be reformulated as a function taking the form $\lambda_n^*(tV,\alpha)$ for $t \in (0,\infty)$ and $\alpha>0$ fixed, arbitrary. The scaling relation gives us immediately that
\begin{equation}
	\label{eq:astscaling}
	t^2 \lambda_j^*(V,\frac{\alpha}{t}) = \lambda_j^*(t^{-N}V,\alpha)
\end{equation}
for all $j\geq 1$, all $V>0$, all $\alpha>0$ and all $t>0$, since \eqref{eq:scaled} holds for every admissible domain $\Omega \subset \mathbb{R}^N$, so that the same must still be true of their infima. Proposition~\ref{prop:starvsalpha} may be reformulated as the following result, which will be useful to us in the sequel. The proof will again be left until Sec.~\ref{sec:append}.

\begin{proposition}
\label{prop:starvst}
Fix $\alpha>0$ and $n \geq 1$. As a function of $V \in (0,\infty)$, $\lambda_n^*(V,\alpha)$ is continuous and strictly monotonically decreasing, with $\lambda_n^*(V,\alpha) \to \infty$ as $V \to 0$ and $\lambda_n^*(V,\alpha) \to 0$ as $V \to \infty$.
\end{proposition}

\subsection{A Wolf--Keller type result}
An immediate consequence of \eqref{eq:invariant} is that both the statement and proof of a number of results that are elementary in the Dirichlet case now become more involved. Of particular relevance for us is the result of Wolf--Keller \cite[Theorem~8.1]{wolf}, that any disconnected domain minimising $\lambda_n$ as in \eqref{eq:ninf} must have as its connected components minimisers of lower numbered eigenvalues. Here, \eqref{eq:invariant} obviously means that we cannot hope to be quite as explicit in our description of any potential minimiser.

\begin{theorem}
\label{th:wkrobin}
Given $V>0$ and $\alpha>0$, suppose that there exists a disconnected domain $\Omega^*$ such that $|\Omega^*| = V$ and $\lambda_n^*(V,\alpha) =\lambda_n (\Omega^*,\alpha)$. For every $1 \leq k \leq n-1$, there will exist a unique pair of numbers $\xi_1, \xi_2 > 1$ (depending on $k,V,\alpha$ and $N$) with $\xi_1^{-N}+\xi_2^{-N}=1$ which solve the problem
\begin{equation}
	\label{eq:pair}
	\begin{split}
	&\min\bigl\{
 \max\{t_1^2\lambda_k^*(V,\frac{\alpha}{t_1}),\,t_2^2\lambda_{n-k}^*(V,\frac{\alpha}{t_2})\}:
	t_1,t_2>1,\,t_1^{-N}+t_2^{-N}=1\bigr\}\\
	=&\min\bigl\{ \max\{\lambda_k^*(t_1^{-N} V,\alpha),\,\lambda_{n-k}^*(t_2^{-N}V,\alpha)\}:
	t_1,t_2>1,\,t_1^{-N}+t_2^{-N}=1\bigr\}.
	\end{split}
\end{equation}

Then we may write
\begin{equation}
	\label{eq:minvalue}
	(\lambda_n^*(V,\alpha))^\frac{N}{2} = \min_{1\leq k\leq n-1}
 \left\{(\lambda_k^*(V,\frac{\alpha}{\xi_1}))^\frac{N}{2}
	+(\lambda_{n-k}^*(V,\frac{\alpha}{\xi_2}))^\frac{N}{2}\right\}.
\end{equation}

 Supposing this minimum to be achieved at some $j$ between $1$ and $n-1$, denoting by $\Omega_1$ and $\Omega_2$ the respective minimisers of $\lambda_j^*(V,\frac{\alpha}{\xi_1})$ and $\lambda_{n-j}^*(V,\frac{\alpha}{\xi_2})$, for this pair $\xi_1(j), \xi_2(j)$ we have
\begin{equation}
	\label{eq:minform}
	\Omega^* = \frac{1}{\xi_1}\Omega_1 \cup \frac{1}{\xi_2}\Omega_2.
\end{equation}
\end{theorem}

Because of \eqref{eq:invariant}, we have to define the scaling constants $\xi_1$ and $\xi_2$ in a somewhat artificial fashion, in terms of a minimax problem (we emphasise that $\xi_1,\xi_2$ will vary with $k$), and cannot link them directly to the optimal values $\lambda_j^*(V,\alpha)$ as would be the direct equivalent of \cite[Theorem~8.1]{wolf}. Otherwise, the proof proceeds essentially as in \cite{wolf}.

\begin{proof}
We start by proving for each fixed $k$ between $1$ and $n-1$ the existence of the pair $\xi_1,\xi_2>1$ as claimed in the theorem. First observe that the equivalence of the two minimax problems in \eqref{eq:pair} follows immediately from \eqref{eq:astscaling}. Now by Proposition~\ref{prop:starvst}, for $V>0$, $\alpha>0$ and $k\geq 1$ all fixed, as a function of $t_1 \in [1,\infty)$, $\lambda_k^*(t_1^{-N}V,\alpha)$ is continuous and strictly monotonically increasing from $\lambda_k^*(V,\alpha)$ at $t_1=1$ to $\infty$ as $t_1 \to \infty$. Moreover, since $t_2$ is determined by $t_1$ via the relation $t_2 = (1- t_1^{-N})^{-1/N}$, we may also consider $\lambda_{n-k}^*(t_2^{-N}V,\alpha)$ as a continuous and strictly monotonically decreasing function of $t_1 \in (1,\infty]$, approaching $\infty$ as $t_1 \to 1$ and $\lambda_{n-k}^*(V,\alpha)$ as $t_1 \to \infty$. That is,
\begin{displaymath}
\lambda_k^*(t_1^{-N} V,\alpha)
\left\{
\begin{aligned}
& < \lambda_{n-k}^*\left((1- t_1^{-N})^{-1/N}V,\alpha\right) & \qquad &\text{if $t_1 \approx 1$}\\
& > \lambda_{n-k}^*\left((1- t_1^{-N})^{-1/N}V,\alpha\right) & &\text{if $t_1$ is large enough,}
\end{aligned}
\right.
\end{displaymath}
with the left hand side value strictly increasing and the right hand side strictly decreasing in $t_1$.
It follows that there exists a unique $t_1 \in (1,\infty)$ such that the two are equal. At this value, which we label as $t_1=:\xi_1$, $t_2 =(1- \xi_1^{-N})^{-1/N}=:\xi_2$, the maximum of the two will be minimised.

Let us now suppose the minimiser $\Omega^*$ of $\lambda_n^*(V,\alpha)$ is a disjoint union $\Omega^* = U_1 \cup U_2$. Since the eigenvalues of $\Omega^*$ are found by collecting and ordering the respective eigenvalues of $U_1$ and $U_2$, there exists $1\leq k \leq n-1$ such that $\lambda_n(\Omega^*,\alpha) = \lambda_k(U_1,\alpha)$. For, if $k=n$, then $U_2$ makes no contribution, so rescaling $U_1$ would strictly decrease $\lambda_n$ by Lemma~\ref{lemma:tcontinuity}, contradicting minimality. A similar argument shows that $\lambda_n(\Omega^*,\alpha)=\lambda_{n-k}(U_2,\alpha)$, since otherwise, by expanding $U_1$ and contracting $U_2$, by Lemma~\ref{lemma:tcontinuity} we could likewise reduce $\lambda_n(\Omega^*,\alpha)$. It is also clear that $\lambda_k(U_1,\alpha)=\lambda_k^*(|U_1|,\alpha)$ and $\lambda_{n-k}(U_2,\alpha) = \lambda_{n-k}^*(|U_2|,\alpha)$, since otherwise we could replace $U_1$ and/or $U_2$ with their respective minimisers and repeat the rescaling argument to reduce $\lambda_n(\Omega^*,\alpha)$. Thus we have shown
\begin{displaymath}
	\lambda_n^*(V,\alpha)=\lambda_n(\Omega^*,\alpha)=\lambda_k(U_1,\alpha)=\lambda_k^*(|U_1|,\alpha)
	=\lambda_{n-k}(U_2,\alpha)=\lambda_{n-k}^*(|U_2|,\alpha).
\end{displaymath}

We now rescale $U_1$ and $U_2$. Let $s_1,s_2>0$ be such that $|s_1 U_1|=|s_2 U_2|=V$. Since $V=|U_1|+|U_2|$, we have $s_1,s_2>1$ and $s_1^{-N}+s_1^{-N}=1$. Now by \eqref{eq:astscaling},
\begin{displaymath}
	(\lambda_k^*(V,\frac{\alpha}{s_1}))^\frac{N}{2}=(\lambda_k(s_1 U_1,\frac{\alpha}{s_1}))^\frac{N}{2}
	=(s_1^{-2}\lambda_k(U_1,\alpha))^\frac{N}{2}=s_1^{-N}(\lambda_n^*(V,\alpha))^\frac{N}{2},
\end{displaymath}
with an analogous statement for $\lambda_{n-k}^*$ and $s_2$. Adding the two, and using that $s_1^{-N}+s_2^{-N}=1$ from the volume constraint,
\begin{displaymath}
	(\lambda_k^*(V,\frac{\alpha}{s_1}))^\frac{N}{2} + (\lambda_{n-k}^*(V,\frac{\alpha}{s_2}))^\frac{N}{2}
	=(\lambda_n^*(V,\alpha))^\frac{N}{2}.
\end{displaymath}
To show that $s_1=\xi_1$ and $s_2=\xi_2$, we simply note that, given this $k$, the unique minimising pair $\xi_1,\xi_2$ is the \emph{only} pair of real numbers for which $\xi_1,\xi_2>1$, $\xi_1^{-N}+\xi_2^{-N}=1$ and for which there is equality ${\xi_1}^2\lambda_k^*(V,\frac{\alpha}{\xi_1}) = {\xi_2}^2\lambda_k^*(V,\frac{\alpha}{\xi_2})$. As $s_1$ and $s_2$ satisfy exactly the same properties, $s_1=\xi_1$ and $s_2=\xi_2$.

Thus we have shown that $\Omega^*$ has the form \eqref{eq:minform}, and \eqref{eq:minvalue} holds for \emph{some} $1 \leq k \leq n-1$. It remains to prove that $\lambda_n^*$ is attained by the minimum over all such $k$. To do so, we choose $1\leq j \leq n-1$ arbitrary, label the solution to \eqref{eq:pair} as $j_1,j_2>1$, and set $\Omega_1^j$ to be the domain of volume $V$ such that
\begin{displaymath}
	\lambda_j(\Omega_1^j,\frac{\alpha}{j_1}) = \lambda_j^*(V,\frac{\alpha}{j_1}),
\end{displaymath}
and analogously for $\Omega_2^j$ and $\lambda_{n-j}^*(V,\alpha/j_1)$. Now set
\begin{displaymath}
	\Omega_j = \frac{1}{j_1}\Omega_1^j \cup \frac{1}{j_2}\Omega_2^j.
\end{displaymath}
It is easy to check that $|\Omega_j|=V$ and that, by choice of $j_1$ and $j_2$,
\begin{displaymath}
	\lambda_j\left(\frac{1}{j_1}\Omega_1^j,\alpha\right)=j_1^2\lambda_j\left(\Omega_1^j,\frac{\alpha}{j_1}\right)
	=j_2^2\lambda_j\left(\Omega_2^j,\frac{\alpha}{j_2}\right)=\lambda_j\left(\frac{1}{j_2}\Omega_2^j,\alpha\right),
\end{displaymath}
meaning that $\lambda_n(\Omega_j,\alpha)$ must be equal to all the above quantities. Moreover, using the minimising properties of $\Omega_1^j$ and $\Omega_2^j$, and that $j_1^{-N}+j_2^{-N}=1$, we have
\begin{displaymath}
	\left[\lambda_n^*(V,\alpha)\right]^\frac{N}{2} \leq \left[\lambda_n(\Omega_j,\alpha)\right]^\frac{N}{2}=
	\left[\lambda_j^*\left(V,\frac{\alpha}{j_1}\right)\right]^\frac{N}{2}
	+\left[\lambda_{n-j}^*\left(V,\frac{\alpha}{j_2}\right)\right]^\frac{N}{2},
\end{displaymath}
proving \eqref{eq:minvalue}.
\end{proof}

\subsection{Asymptotic behaviour of the optimal values}

Another consequence of \eqref{eq:invariant} is that any eigenvalue $\lambda_n (t\Omega,\alpha)$ grows more slowly than $\lambda_n(\Omega)^{2/N}$ as $t \to 0$. It is thus intuitively reasonable that we might expect any optimising domain to have a greater number of connected components than its Dirichlet counterpart. Indeed, recalling the variational characterisation of $\lambda_n$, it is not surprising that increasing the size of the boundary in such a fashion carries a fundamentally smaller penalty for the eigenvalues. As was noted in \cite[Sec.~3]{kenn2}, the domain given by the disjoint union of $n$ equal balls of volume $V/n$, call it $B_n$, is likely to play a prominent r\^ole in the study of $\lambda_n^*$ for sufficiently small positive values of $\alpha$. Here we go further and observe that, simply by estimating $\lambda_1(B_n,\alpha)$, we can already obtain quite a strong estimate on the behaviour of $\lambda_n^*$ with respect to $n$, for any $\alpha>0$.

In fact, the following theorem, which again may be seen as an immediate consequence of \eqref{eq:invariant}, shows that we have $\lambda_n^*= \so(n^{1/N+\varepsilon})$ as $n \to \infty$ (for any $V,\alpha,\varepsilon>0$), a fundamental divergence from the Weyl asymptotics $\lambda_n(\Omega,\alpha) = \bo(n^{2/N})$ for any fixed domain $\Omega \subset \mathbb{R}^N$. It is unclear whether this is optimal.

\begin{theorem}
	\label{th:nballs}
	Given $V>0$ and $n \geq 1$, let $B_n$ denote the domain of volume $V$ consisting of $n$ equal
balls of radius $r= (V/n\omega_N)^{1/N}$. Then, for every $\alpha>0$,
	\begin{equation}
	\label{eq:nballs}
	\lambda_n^*(V,\alpha) \leq \lambda_n(B_n,\alpha) \leq N\alpha \left(\frac{n\omega_N}{V}\right)^\frac{1}{N}.
	\end{equation}
\end{theorem}

\begin{proof}
Since $\lambda_n(B_n,\alpha)=\lambda_1(B_n,\alpha)$, it certainly suffices to estimate the latter, that is, to estimate the first eigenvalue of a ball of volume $V/n$ and radius $r = (V/n\omega_N)^{1/N}$. Using concavity of $\lambda_1$ with respect to $\alpha>0$ (Lemma~\ref{lemma:continuity}), we estimate this from above by its tangent line at $\alpha=0$ (see Remark~\ref{rem:neumann}). Since a ball of radius $r$ has volume $r^N \omega_N$ and surface measure $Nr^{N-1}\omega_N$, \eqref{eq:alphaderivative} at $\alpha = 0$ gives
\begin{displaymath}
	\lambda_1(B_n,\alpha)\leq \lambda_1'(B_n,0)\alpha = \frac{\sigma(\partial B_n)}{|B_n|}\alpha = N r^{-1} \alpha.
\end{displaymath}
Substituting the value $r = (V/n\omega_N)^{1/N}$ yields \eqref{eq:nballs}.
\end{proof}

\subsection{The optimal gap}

Adapting an argument of Colbois and El~Soufi \cite{colbois} for the Dirichlet case, we may also estimate the dimensionally appropriate difference $(\lambda_{n+1}^*)^{N/2}-(\lambda_n^*)^{N/2}$ for each positive $V$ and $\alpha$, which we do in Theorem~\ref{th:gapbound}. Such an estimate serves two purposes, giving both a practical means to test the plausibility of numerical estimates, and a theoretical bound on eigenvalue gaps. In particular, this complements Theorem~\ref{th:nballs} by showing that the optimal gap tends to $0$ as $n \to \infty$, albeit not necessarily uniformly in $\alpha>0$ (see Corollary~\ref{cor:growth}).

The idea is to take the optimising domain $\Omega^*$ for $\lambda_n^*$, add to it a ball $B$ whose first eigenvalue also equals $\lambda_n^*$ and then rescale to obtain a ``test domain" for $\lambda_{n+1}^*$. Although the scaling issue \eqref{eq:invariant} makes the new behaviour possible, it also causes obvious complications, and so we cannot obtain as tight a result as in \cite{colbois}. Instead, we will give two slightly different estimates. The first, \eqref{eq:comp}, is tighter but more abstruse, and will be used for computational verification, the second one \eqref{eq:explicit} being more explicit, although only in the first case does the bound converge to $0$ with $n$ (see Remark~\ref{rem:gapbound} and Corollary~\ref{cor:growth}).

\begin{remark}
\label{rem:dirichlet}
The result that the optimal gap tends to $0$ as $n\to\infty$ is perhaps {\it{a priori}} surprising, and raises the question of whether such a result might also be true in the Dirichlet case. Unfortunately, our method tells us nothing about the latter, as it rests entirely on the scaling property \eqref{eq:scaled}. For each fixed $n\geq 1$, our bound \eqref{eq:comp} is of the form $(\lambda_{n+1}^*)^{N/2}-(\lambda_n^*)^{N/2} \leq C(\lambda_1(B_1,c_n\alpha))^{N/2}$ for appropriate constants $C=C(N,V)$ and $c_n=c(N,V,n)$. The idea is then to show (using \eqref{eq:scaled}) that $c_n \to 0$ as $n \to \infty$. But fixing $n\geq 1$ and letting $\alpha \to \infty$, we recover the bound from \cite{colbois} of the form $C(N,V) \lambda_1^D(B_1)$, uniform in $n \geq 1$. There is no evidence to suggest the latter result could be improved.
\end{remark}

 The bound \eqref{eq:comp} relies on the following auxiliary lemma, needed for a good estimate of the constant $c_n$ mentioned in Remark~\ref{rem:dirichlet}.

\begin{lemma}
\label{lemma:ball}
Fix $V>0$ and $\alpha>0$, let $\lambda_n^*=\lambda_n^*(V,\alpha)$ be as in \eqref{eq:ninf},
and denote by $B=B(0,r)$ the ball centred at $0$ of radius $r>0$. There exists a unique value
of $r>0$ such that
\begin{equation}
	\label{eq:ball}
	\lambda_1(B,\Bigl(\frac{V}{V+|B|}\Bigr)^{\frac{1}{N}} \alpha)=\lambda_n^*(V,\alpha).
\end{equation}
The corresponding ball $B$ satisfies
\begin{equation}
	\label{eq:ballsize}
	|B|\leq \min\left\{V,\, \omega_N \left(\frac{j_{\frac{N}{2}-1,1}}{\sqrt{\lambda_n^*}}\right)^N
 \right\},
\end{equation}
where $\omega_N$ is the $N$-dimensional volume of the unit ball in $\mathbb{R}^N$ and $j_{\frac{N}{2}-1,1}$ the first zero of the Bessel function $J_{\frac{N}{2}-1}$ of the first kind.
\end{lemma}

\begin{proof}
Consider the left hand side of \eqref{eq:ball} as a function of $|B| \in (0,\infty)$. An increase in $|B|$ both increases the volume of the domain and decreases the Robin parameter. By Lemma~\ref{lemma:tcontinuity}, the combined effect must be to decrease $\lambda_1$ continuously and strictly monotonically, the latter implying there can be at most one value of $|B|$ giving equality in \eqref{eq:ball}. Now note that as $|B|\to 0$, since $V/(V+|B|)$ is bounded from below away from zero, $\lambda_1(B,(V/(V+|B|))^{(1/N)} \alpha) \to \infty$, while if $|B|=V$,
\begin{displaymath}
	\lambda_1(B, (1/2)^\frac{1}{N}\alpha)< \lambda_1 (B,\alpha)
	\leq \lambda_1(\Omega^*,\alpha) \leq \lambda_n(\Omega^*,\alpha)=\lambda_n^*(V,\alpha),
\end{displaymath}
where $\Omega^*$ is the minimising domain for \eqref{eq:ninf}, and the second inequality follows from the Rayleigh--Faber--Krahn inequality for Robin problems \cite[Theorem~1.1]{buda}. Hence there must be a value of $|B|$ in $(0,V)$ for which there is equality in \eqref{eq:ball}. To show that the other bound in \eqref{eq:ballsize} also holds, we consider $B_r$, the ball of radius $r=j_{\frac{N}{2}-1,1}/\sqrt{\lambda_n^*}$, where $j_{p,q}$ is the $q$th zero of the Bessel function $J_p$ of the first kind of order $p$. Then we have
\begin{displaymath}
	\lambda_1(B_r,\Bigl(\frac{V}{V+|B_r|}\Bigr)^{\frac{1}{N}} \alpha) < \lambda_1 (B_r,\alpha) <
 \lambda_1^D(B_r)
	=\lambda_n^*,
\end{displaymath}
where $\lambda_1^D(B_r)$ is the first Dirichlet eigenvalue of $B_r$ (cf.~Lemma~\ref{lemma:continuity}), and the last equality follows from our choice of $r$. This implies our desired $B$ must have radius less than $r$, giving us \eqref{eq:ballsize}.
\end{proof}

\begin{theorem}
	\label{th:gapbound}
	Fix $V>0$ and $\alpha>0$ and let $\lambda_n^*=\lambda_n^*(V,\alpha)$ be as in \eqref{eq:ninf}. Let $B^* = B^*(V,n,\alpha)$ be the ball satisfying the conclusions of Lemma~\ref{lemma:ball}, and let $B_1$ denote the ball of unit radius and $\omega_N$ its $N$-dimensional volume. Then
\begin{equation}
	\label{eq:comp}
	\left(\lambda_{n+1}^* \right)^\frac{N}{2} - \left(\lambda_n^* \right)^\frac{N}{2} \leq \frac{\omega_N}{V}
	\left[\lambda_1\left(B_1, \left(\frac{V|B^*|}{V+|B^*|}\right)^\frac{1}{N} \omega_N^{-\frac{1}{N}}\alpha\right)
	\right]^\frac{N}{2}
\end{equation}
and, weaker but more explicitly,
\begin{equation}
	\label{eq:explicit}
	\left(\lambda_{n+1}^* \right)^\frac{N}{2} - \left(\lambda_n^* \right)^\frac{N}{2} < \frac{\omega_N}{V}
	\left[\lambda_1\left(B_1, \left(\frac{V}{\omega_N}\right)^\frac{1}{N}\alpha\right)\right]^\frac{N}{2}.
\end{equation}
\end{theorem}

Our proof will show, in a manner analogous to the Dirichlet case \cite{colbois}, that given \emph{any} domain $\widetilde\Omega$ we can find another domain $\widehat\Omega$ of the same volume $V$ such that
\begin{displaymath}
	\left[\lambda_{n+1}(\widehat\Omega,\alpha)\right]^{\frac{N}{2}}-
	\left[\lambda_n(\widetilde\Omega,\alpha)\right]^{\frac{N}{2}} \leq
	\frac{\omega_N}{V}\left[\lambda_1(B_1,t\alpha)\right]^\frac{N}{2}
\end{displaymath}
for some appropriate $t\in (0,1)$ inversely proportional to $\lambda_n(\widetilde\Omega,\alpha)$. We omit the details.

\begin{proof}
Given $V,\alpha>0$, $B^*$ and $\lambda_n^*$ as in the statement of the theorem, we assume for simplicity that $\lambda_n^*$ is minimised by $\Omega^*$. Let $\widetilde \Omega:= \Omega^* \cup B^*$ (disjoint union), set
\begin{displaymath}
	t^N:=\frac{V}{V+|B^*|}
\end{displaymath}
and consider the problem \eqref{eq:robin} on $\widetilde \Omega$, with boundary parameter
\begin{displaymath}
	\hat\alpha:=\left\{
	\begin{aligned}
	&\alpha \qquad\qquad &\text{on $\partial \Omega^*$}\\
	&t\alpha \qquad\qquad &\text{on $\partial B^*$.}
	\end{aligned}
	\right.
\end{displaymath}
Then by choice of $B^*$ and definition of $t$,
\begin{displaymath}
	\lambda_{n+1}(\widetilde\Omega,\hat\alpha)=\lambda_n(\Omega^*,\alpha)=\lambda_1(B^*,t\alpha).
\end{displaymath}
Let us now rescale $\widetilde\Omega$ to $t\widetilde\Omega$. If we set
\begin{displaymath}
	\tilde\alpha:=\left\{
	\begin{aligned}
	&\alpha/t \qquad\qquad &\text{on $\partial (t\Omega^*)$}\\
	&\alpha \qquad\qquad &\text{on $\partial (t B^*)$,}
	\end{aligned}
	\right.
\end{displaymath}
we have
\begin{displaymath}
	\lambda_{n+1}(t\widetilde\Omega,\widetilde\alpha)=\frac{1}{t^2}\lambda_{n+1}(\widetilde\Omega,\hat\alpha).
\end{displaymath}
Since $|t\widetilde\Omega|=V$,
\begin{equation}
	\label{eq:loss}
	\lambda_{n+1}^* \leq \lambda_{n+1}(t\widetilde\Omega,\alpha)\leq
	\lambda_{n+1}(t\widetilde\Omega,\tilde\alpha),
\end{equation}
the latter inequality holding since $t<1$ so that $\alpha\leq \tilde\alpha=\tilde\alpha(x)$ at every point of $\partial(t\widetilde\Omega)$. Hence
\begin{displaymath}
	\lambda_{n+1}^* \leq \frac{1}{t^2}\lambda_{n+1}(\widetilde\Omega,\hat\alpha)
	=\frac{1}{t^2}\lambda_n(\Omega^*)=\frac{1}{t^2}\lambda_1(B^*,t\alpha).
\end{displaymath}
Raising everything to the power of $N/2$, and subtracting $\left(\lambda_n^*\right)^{N/2} = \left(\lambda_1(B^*,t\alpha)\right)^{N/2}$,
\begin{displaymath}
	\left(\lambda_{n+1}^* \right)^\frac{N}{2} - \left(\lambda_n^* \right)^\frac{N}{2} \leq
	 \left(\frac{1}{t^N}-1\right)
	\left(\lambda_1(B^*,t\alpha)\right)^\frac{N}{2}
\end{displaymath}
Recalling the definition of $t$, we have $1/t^N-1 = |B^*|/V$. We will now rescale $B^*$, replacing it with a ball of unit radius. That is, letting $B^*$ have radius $r>0$, so that $|B^*|=r^N \omega_N$, and letting $B_1$ denote the ball of centre $0$ and radius $1$,
\begin{displaymath}
	\left(\lambda_1(B^*,t\alpha)\right)^\frac{N}{2}=\left(r^{-2}\lambda_1(B_1,rt\alpha)\right)^\frac{N}{2}
	=r^{-N}\left(\lambda_1(B_1,rt\alpha)\right)^\frac{N}{2}.
\end{displaymath}
Writing
\begin{equation}
	\label{eq:rt}
	rt =  \left(\frac{V|B^*|}{V+|B^*|}\right)^\frac{1}{N} \omega_N^{-\frac{1}{N}}
\end{equation}
and $|B^*|/V = r^N \omega_N/V$, we obtain
\begin{displaymath}
	\left(\lambda_{n+1}^* \right)^\frac{N}{2} - \left(\lambda_n^* \right)^\frac{N}{2} \leq \frac{r^N
	\omega_N}{V} r^{-N}
	\left[\lambda_1\left(B_1,\left(\frac{V|B^*|}{V+|B^*|}\right)^\frac{1}{N} \omega_N^{-\frac{1}{N}}\alpha\right)
	\right]^\frac{N}{2},
\end{displaymath}
which is \eqref{eq:comp}. To remove the explicit dependence on $B^*$ and thus obtain \eqref{eq:explicit}, we simplify the expression \eqref{eq:rt} using the crude bounds $|B^*|<V$ in the numerator and $|B^*|>0$ in the denominator, giving $rt<V^{(1/N)} \omega_N^{-(1/N)}$. Monotonicity of $\lambda_1$ with respect to the Robin parameter, Lemma~\ref{lemma:continuity}, now gives \eqref{eq:explicit}.
\end{proof}

\begin{remark}
\label{rem:gapbound}
(i)  The bound \eqref{eq:explicit} is as explicit as possible for the Robin problem, $\lambda_1(B_1,\beta)$ being given as the square of the first positive solution of the transcendental equation
\begin{displaymath}
	\frac{\beta}{\sqrt{\lambda_1}} = -\frac{J_{\frac{N}{2}}(\sqrt \lambda_1)}
	{J_{\frac{N}{2}-1}(\sqrt \lambda_1)},
\end{displaymath}
where $J_p$ denotes the $p$th Bessel function of the first kind.

(ii) In the Dirichlet equivalent of Theorem~\ref{th:gapbound}, the bound is optimal exactly for those values of $n$ for which the optimising domain for $\lambda_{n+1}^*$ is obtained by adding an appropriate ball to the minimiser of $\lambda_n$, which is believed to be true only when $n=1$. In our case, since everything converges to its Dirichlet counterpart as $\alpha \to \infty$, \eqref{eq:comp} and \eqref{eq:explicit} are at least asymptotically sharp for $n=1$. Moreover, for every $n\geq 1$ the bounds converge to zero as $\alpha \to 0$ (as we would hope given that $\lambda_n^* \to 0$ as $\alpha\to 0$ for every $n\geq 1$). However, even for $n=1$, the scaling issue makes it essentially impossible to obtain a precise bound for any particular $\alpha>0$. Taking $N=2$ for simplicity, we have $\lambda_2^*=2\lambda_1^*$ in the Dirichlet case, but in the Robin case $\lambda_2^*(\alpha) < 2\lambda_1^*(\alpha)$ for all $\alpha>0$, since, denoting by $B$ the ball that minimises $\lambda_1^*(\alpha)$, we have $\lambda_2^*(\alpha) = \lambda_1(2^{-1/2}B,\alpha) = 2\lambda_1(B,\alpha/2) < 2\lambda_1(B,\alpha)=2\lambda_1^*(\alpha)$, where we have used the fact that the minimiser of $\lambda_2^*(\alpha)$ is the union of two equal balls \cite{kenn1}, the scaling relation \eqref{eq:scaled}, strict monotonicity of $\lambda_1(\Omega, \alpha)$ in $\alpha$ (Lemma~\ref{lemma:continuity}) and the Rayleigh--Faber--Krahn inequality for Robin problems \cite{buda}. Our bound in \eqref{eq:comp} is, in this case, also smaller than $2\lambda_1^*(\alpha)$ for all $\alpha>0$. However, constructing an estimate that involves rescaling domains in this fashion will always tend to introduce some error (as happens at \eqref{eq:loss} in our case), as we can never write down explicitly the change in the eigenvalues caused by introducing the scaling parameter $t$ into the boundary parameter.
\end{remark}

We now prove the aforementioned result, a complement to Theorem~\ref{th:nballs}, that the dimension-normalised gap $(\lambda_{n+1}^*)^{N/2}-(\lambda_n^*)^{N/2}$ approaches zero as $n$ goes to $\infty$ for every fixed positive value of $V$ and $\alpha$. The proof will combine \eqref{eq:comp} with \eqref{eq:ballsize}. In the process, we also obtain a growth estimate on $\lambda_n^*$, but this turns out to be weaker than the one found directly in Theorem~\ref{th:nballs}. We include the proof of the latter anyway, as both an alternative method and to illustrate the principle.

\begin{corollary}
\label{cor:growth}
	For $V,\alpha>0$ fixed, as $n\to\infty$ we have
	\begin{displaymath}
		\left[\lambda_{n+1}^*(V,\alpha)\right]^\frac{N}{2} -
\left[\lambda_n^*(V,\alpha)\right]^\frac{N}{2} \to 0
	\end{displaymath}
	and, for every $\varepsilon>0$,
	\begin{displaymath}
		\lambda_n^*(V,\alpha) = \so(n^{\frac{4}{3N}+\varepsilon}).
	\end{displaymath}
\end{corollary}

\begin{proof}
Estimating $|B^*|$ from above by \eqref{eq:ballsize} and from below by $0$ in the bound \eqref{eq:comp} gives us
\begin{displaymath}
	\left(\lambda_{n+1}^* \right)^\frac{N}{2} - \left(\lambda_n^* \right)^\frac{N}{2} < \frac{\omega_N}{V}
	\left[\lambda_1\left(B_1, \frac{j_{\frac{N}{2}-1,1}}{\sqrt{\lambda_n^*}}\alpha\right)\right]^\frac{N}{2}.
\end{displaymath}
 Now, noting that $\lambda_1(B_1,\,\cdot\,)$ is concave in its second argument, we estimate it from above by the
 corresponding value of its tangent line at $\alpha=0$, namely
$T(\beta)=\lambda_1'(B_1,0)\beta$. Since $\lambda_1'
 (B_1,0) = \sigma(\partial B_1)/\omega_N$ (see Remark~\ref{rem:neumann}),
\begin{equation}
	\label{eq:growthcor}
	\left(\lambda_{n+1}^* \right)^\frac{N}{2} - \left(\lambda_n^* \right)^\frac{N}{2} < V^{-1}
	 {\omega_N}^{1-\frac{N}{2}}
	\left(\sigma(\partial B_1)\alpha j_{\frac{N}{2}-1,1}\right)^\frac{N}{2}\left(\lambda_n^*
	 \right)^{-\frac{N}{4}}.
\end{equation}
Observe that, for fixed $V,\alpha>0$, the right hand side of \eqref{eq:growthcor} converges to $0$ as $n\to\infty$, proving the first assertion of the corollary. To see the growth bound, label the coefficient of $\left(\lambda_n^* \right)^{-N/4}$ in \eqref{eq:growthcor} as $C(N,V,\alpha)$. Summing over $n$, the left hand side telescopes to give
\begin{equation}
	\label{eq:sumbound}
	\left(\lambda_{n+1}^* \right)^\frac{N}{2} < \left(\lambda_{1}^* \right)^\frac{N}{2}+ C(N,V,\alpha)
	\sum_{k=1}^n \left(\lambda_k^* \right)^{-\frac{N}{4}}.
\end{equation}
Suppose now that for some $\gamma\in\mathbb{R}$, $(\lambda_n^*)^{N/2} \neq \bo(n^\gamma)$ as $n \to \infty$, that is, $\limsup_{n\to\infty} (\lambda_n^*)^{N/2}/n^\gamma = \infty$. Since $\lambda_n^*$ is increasing in $n\geq 1$, a standard argument from elementary analysis shows that in fact $\lim_{n\to\infty} (\lambda_n^*)^{N/2}/n^\gamma = \infty$, that is, for all $C_0>0$, there exists $n_0 \geq 1$ such that $(\lambda_n^*)^{N/2} \geq C_0 n^\gamma$ for all $n \geq n_0$. Hence for $C_0>0$ fixed, for all $n \geq n_0$ we have
\begin{equation}
\label{eq:growthcontra}
\sum_{k=1}^n (\lambda_k^*)^{-\frac{N}{4}}\leq \sum_{k=1}^{n_0-1}(\lambda_k^*)^{-\frac{N}{4}}+C_0
	\sum_{k=n_0}^n k^{-\frac{\gamma}{2}}\leq C_1+C_0 \sum_{k=1}^n k^{-\frac{\gamma}{2}},
\end{equation}
where $C_1 = \sum_{k=1}^{n_0-1}(\lambda_k^*)^{-N/4}$ depends only on $N,V,\alpha$ and $n_0$, that is, $C_0$. Now we observe
\begin{equation}
	\label{eq:harmonics}
	\bo(\sum_{k=1}^n k^{-\gamma/2}) = \left\{
	\begin{aligned}
	& \bo(n^{1-\frac{\gamma}{2}}) & \qquad \quad &\text{if $\gamma \in (0,2)$}\\
	& \bo(\ln n) & &\text{if $\gamma = 2$,}
	\end{aligned}
	\right.
\end{equation}
(use $\sum_{k=1}^n k^{-s} \sim \int_1^n x^{-s}dx$ if $s \leq 1$), while $\sum_{k=1}^n k^{-\gamma/2} \leq 1+\int_1^{n-1} x^{-\gamma/2}\,dx = 1+(\gamma/2-1)^{-1}$ for all $n \geq 1$, if $\gamma>2$. In particular, combining \eqref{eq:sumbound} and \eqref{eq:growthcontra}, fixing $C_0>0$ and a corresponding $n_0 \geq 1$ arbitrary, for all $n\geq n_0$, we have
\begin{displaymath}
	(\lambda_{n+1}^*)^\frac{N}{2} < (\lambda_1^*)^\frac{N}{2}+C(N,V,\alpha)(C_1+C_0
	\sum_{k=1}^n k^{-\frac{\gamma}{2}}),
\end{displaymath}
which we rewrite as
\begin{equation}
	\label{eq:obound}
	(\lambda_n^*)^\frac{N}{2} < C_2 + C_3 \sum_{k=1}^n k^{-\frac{\gamma}{2}}
\end{equation}
for all $n \geq n_0$, where $C_2,C_3>0$ and $n_0\geq 1$ depend only on $N,V,\alpha$ and the free choice $C_0>0$. Recalling that \eqref{eq:obound} holds under the assumption $(\lambda_n^*)^{N/2} \neq \bo(n^\gamma)$ as $n \to \infty$ and using \eqref{eq:harmonics}, this gives us an immediate contradiction if $\gamma>2/3$, forcing $(\lambda_n^*)^{N/2} = \so(n^{2/3+\varepsilon})$ for all $\varepsilon>0$.
\end{proof}

\subsection{Dependence of $\lambda_{n}(t\Omega,\alpha)$ on $t$ and $\alpha$}
\label{sec:append}

We will now give some appendiceal, but important, results on the behaviour of the Robin eigenvalues $\lambda_n (\Omega,\alpha)$ with respect to homothetic changes in $\Omega$ or $\alpha$. Although the material is folklore, we have included a proof as it seems difficult to find one explicitly. We will also give the proofs of the corresponding statements for $\lambda_n^*(V,\alpha)$, namely Propositions~\ref{prop:starvsalpha} and \ref{prop:starvst}.

\begin{lemma}
\label{lemma:continuity}
For a given bounded, Lipschitz domain $\Omega \subset \mathbb{R}^N$, and $n \geq 1$, $\lambda_n(\Omega,\alpha)$ is an absolutely continuous and strictly increasing function of $\alpha \in [0,\infty)$, which is differentiable almost everywhere in $(0,\infty)$. Where it exists, its derivative is given by
\begin{equation}
	\label{eq:alphaderivative}
	\frac{d}{d\alpha} \lambda_n(\Omega,\alpha) = \frac{\|u\|_{2,\partial\Omega}^2}{\|u\|_{2,\Omega}^2},
\end{equation}
where $u\in H^1(\Omega)$ is any eigenfunction associated with $\lambda_n(\Omega,\alpha)$. In addition, when $n=1$, $\lambda_1(\Omega,\alpha)$ is concave, with $\lambda_1(\Omega,\alpha)\leq \lambda_1^D(\Omega)$, the first Dirichlet eigenvalue of $\Omega$, and if $\Omega$ is connected, then $\lambda_1(\Omega,\alpha)$ is analytic in $\alpha \geq 0$.
\end{lemma}

\begin{remark}
\label{rem:neumann}
The formula \eqref{eq:alphaderivative} is always valid when $n=1$ and $\alpha=0$, for any bounded, Lipschitz $\Omega \subset \mathbb{R}^N$. In this case, since this corresponds to the Neumann problem and the first eigenfunction is always constant, \eqref{eq:alphaderivative} simplifies to
\begin{equation}
\label{eq:neumannderivative}
	\frac{d}{d\alpha} \lambda_1(\Omega,\alpha)\Big|_{\alpha=0} = \frac{\sigma(\partial\Omega)}{|\Omega|},
\end{equation}
a purely geometric property of $\Omega$. The equation \eqref{eq:neumannderivative} (which can be obtained from a trivial modification of our proof of Lemma~\ref{lemma:continuity}) is reasonably well known, and proofs may also be found in \cite{giorgi,lacey}, for example.
\end{remark}

\begin{lemma}
\label{lemma:tcontinuity}
Given $\Omega \subset \mathbb{R}^N$ bounded and Lipschitz, $n \geq 1$ and $\alpha>0$, $\lambda_n(t\Omega,\alpha)$ is a continuous and strictly decreasing function of $t \in (0,\infty)$. If $\frac{d}{d\beta} \lambda_n(\Omega,\beta)$ exists at $\beta=t\alpha>0$, then so does
\begin{equation}
	\label{eq:tderivative}
	\frac{d}{dt}\lambda_n(t\Omega,\alpha)
	= -\frac{1}{t^3}\left(\frac{\|\nabla v\|_{2,\Omega}^2}{\|v\|_{2,\Omega}^2}-\lambda_n(\Omega,t\alpha)\right),
\end{equation}
where $v\in H^1(\Omega)$ is any eigenfunction associated with $\lambda_n(\Omega,t\alpha)$.
\end{lemma}

\begin{proof}[Proof of Lemma~\ref{lemma:continuity}]
For the first statement, we note that (weak) monotonicity and, when $n=1$, concavity, are immediate from the minimax formula for $\lambda_n$ \cite[Chapter~VI]{courant}. We can also derive continuity directly from that formula, or use the general theory from \cite[Sec.~VII.3, 4]{kato}. That is, the form associated with \eqref{eq:robin} is
\begin{displaymath}
	Q_\alpha(u)=\int_\Omega |\nabla u|^2\,dx+\int_{\partial\Omega}\alpha u^2\,d\sigma,
\end{displaymath}
which is analytic in $\alpha \in \mathbb{R}$ for each $u \in H^1(\Omega)$. Hence the associated family of self-adjoint operators is holomorphic of type (B) in the sense of Kato. It follows that each eigenvalue depends locally analytically on $\alpha$, with only a countable number of ``splitting points", that is, crossings of curves of eigenvalues, including the possibility of splits in multiplicities. In our case, for each $\lambda_n(\alpha)$, the number of such points will certainly be locally finite in $\alpha$. In particular, this means $\lambda_n(\alpha)$ consists locally of a finite number of smooth curves intersecting each other transversally, so it is absolutely continuous in the sense of \cite[Ch.~7]{rudin}. (Throughout this lemma we drop the $\Omega$ argument, as it is fixed.) If $n=1$ and $\Omega$ is connected, then $\lambda_1(\alpha)$ has multiplicity one for all $\alpha \geq 0$ and hence no splitting points, so that it is analytic.

Since \cite{kato} also implies that the associated eigenprojections converge, given any non-splitting point $\alpha$ (at which $\lambda_n$ is analytic), $\alpha_k\to\alpha$ and any eigenfunction $u$ associated with $\lambda_n(\alpha)$, we can find eigenfunctions $u_k$ of $\lambda_n(\alpha_k)$ such that $u_k \to u$ in $L^2(\Omega)$. We now use a standard argument to show that in fact $u_k \to u$ in $H^1(\Omega)$. Denote by $\|v\|_*$ the norm on $H^1(\Omega)$ given by $(\|\nabla v\|_{2,\Omega}^2+\alpha\|v\|_{2,\partial\Omega}^2)^{1/2}$, equivalent to the standard one, and assume the eigenfunctions are normalised so that $\|u\|_{2,\Omega} = \|u_k\|_{2,\Omega} = 1$ for all $k \geq 1$. Then
\begin{displaymath}
\begin{split}
\|u-u_k\|_*^2 &= \int_\Omega |\nabla u|^2+|\nabla u_k|^2 - 2\nabla u\cdot \nabla u_k\,dx
	+\int_{\partial\Omega}\alpha (u^2+u_k^2 - 2 u u_k)\,d\sigma\\
	&=\lambda_n(\alpha)+\lambda_n(\alpha_k)-2\lambda_n(\alpha)\int_\Omega u u_k\,dx + (\alpha-\alpha_k)
	\int_{\partial\Omega}u_k^2\,d\sigma,
\end{split}
\end{displaymath}
making repeated use of \eqref{eq:weak}. Now we have $\alpha_k \to \alpha$ by assumption, while we may use the crude but uniform bound $\int_{\partial\Omega} u_k^2\,d\sigma \leq \lambda_n(\alpha_k)/\alpha_k$ for $\alpha_k \to \alpha>0$ bounded away from zero. Meanwhile, by H\"older's inequality,
\begin{displaymath}
\left|\int_\Omega uu_k\,dx - \int_\Omega u^2\,dx\right|\leq \|u\|_{2,\Omega}\|u-u_k\|_{2,\Omega} \longrightarrow 0
\end{displaymath}
as $k \to \infty$, since we know $u_k \to u$ in $L^2(\Omega)$, meaning $\int_\Omega uu_k\,dx \to 1$ due to our normalisation. Hence, since $\lambda_n(\alpha_k) \to \lambda_n(\alpha)$ also,
\begin{displaymath}
\|u-u_k\|_*^2 =\lambda_n(\alpha)+\lambda_n(\alpha_k) - 2\lambda_n(\alpha)\int_\Omega u u_k\,dx + (\alpha-\alpha_k)\int_{\partial\Omega}
	u_k^2\,d\sigma \longrightarrow 0,
\end{displaymath}
proving $u_k \to u$ in both the $\|\,.\,\|_*$-norm and hence also in the usual $H^1$-norm.

Let us now compute the derivative of $\lambda_n(\alpha)$ at any non-splitting point. Suppose $\alpha<\beta$ are two different boundary parameters, with associated eigenfunctions $u,v \in H^1(\Omega)$, respectively. Then, using the weak form of $\lambda_n$, provided $u$ and $v$ are not orthogonal in $L^2(\Omega)$, we get immediately that
\begin{displaymath}
	\lambda_n(\beta)-\lambda_n(\alpha)=(\beta-\alpha)\frac{\int_{\partial\Omega} uv\,d\sigma}
	{\int_\Omega uv\,dx}.
\end{displaymath}
Now divide through by $\beta-\alpha$ and let $\beta\to\alpha$. Since we have already seen that this forces $v\to u$ in $H^1(\Omega)$ (also implying that they are not orthogonal in $L^2(\Omega)$ for $\beta$ close to $\alpha$), this gives \eqref{eq:alphaderivative}. In particular, since this is strictly positive, and valid except on a countable set of $\alpha$, we conclude $\lambda_n$ is strictly increasing. Note that even at splitting points, we may still compute the left and right derivatives via this method; we see that it is the change in multiplicity (leading to extra eigenfunctions giving different values of \eqref{eq:alphaderivative}) that causes these derivatives to disagree.

That $\lambda_1(\Omega,\alpha) \leq \lambda_1^D(\Omega)$ is an immediate consequence of the minimax formulae and the inclusion of the form domains $H^1_0(\Omega) \subset H^1(\Omega)$. Strict inequality is immediate, since $\lambda_1(\Omega,\alpha)$ is strictly increasing in $\alpha>0$.
\end{proof}

\begin{proof}[Proof of Lemma~\ref{lemma:tcontinuity}]
Since $\lambda_n(t\Omega,\alpha) = t^{-2}\lambda_n(\Omega,\alpha t)$, differentiability of the former at $t$ is equivalent to differentiability of the latter at $\beta=\alpha t$, and
\begin{displaymath}
	\frac{d}{dt}\lambda_n(t\Omega,\alpha) = \frac{d}{dt}\left(t^{-2}\lambda_n(\Omega,\alpha t)\right)
	= -2t^{-3}\lambda_n(\Omega,\alpha t)+\alpha t^{-2}\frac{d}{d(\alpha t)}\lambda_n(\Omega,\alpha t).
\end{displaymath}
Using \eqref{eq:alphaderivative} and simplifying yields \eqref{eq:tderivative}. Continuity of $\lambda_n(t\Omega,\alpha)$ at every $t>0$ follows immediately from the identity
\begin{displaymath}
\lambda_n(s\Omega,\alpha)-\lambda_n(t\Omega,\alpha)=s^{-2}\lambda_n(\Omega,s\alpha)-t^{-2}\lambda(\Omega,t\alpha)
\end{displaymath}
together with continuity of $\lambda_n(\Omega,t\alpha)$ in $t$, so that $s^{-2}\lambda_n(\Omega,s\alpha) \to t^{-2}\lambda(\Omega,t\alpha)$ as $s \to t$.

Finally, observe that \eqref{eq:tderivative}, holding almost everywhere, also confirms the strict monotonicity of $\lambda_n(t\Omega,\alpha)$ with respect to $t>0$. Also note that even at points of discontinuity, as in Lemma~\ref{lemma:continuity}, we can again compute left and right derivatives, which may disagree due to the change in multiplicity.
\end{proof}

\begin{proof}[Proof of Proposition~\ref{prop:starvsalpha}]
It follows immediately from the definition of $\lambda_n^*(V,\alpha)$ as an infimum and the properties of $\lambda_n(\Omega,\alpha)$ given in Lemma~\ref{lemma:continuity} that $\lambda_n^*(V,\alpha)$ is strictly increasing and right continuous in $\alpha$. Indeed, given $\alpha_0 \in [0,\infty)$ and $\alpha>\alpha_0$, if $\Omega_0^*$ is a minimising domain so that $\lambda_n^*(V,\alpha_0) = \lambda_n(\Omega_0^*,\alpha_0)$, and $\alpha>0$, then $0\leq \lambda_n^*(V,\alpha) - \lambda_n(V,\alpha_0) < \lambda_n^*(\Omega_0^*,\alpha) - \lambda_n(\Omega_0^*,\alpha_0) \to 0$ as $\alpha \to \alpha_0$. For strict monotonicity, if $0 \leq \alpha_0 < \alpha$, we let $\Omega^*$ be such that $\lambda_n^*(V,\alpha) = \lambda_n(\Omega^*,\alpha) > \lambda_n (\Omega^*, \alpha_0) \geq \lambda_n^* (V,\alpha_0)$.

Left continuity is harder. We use the property that for each fixed domain $\Omega$, by \eqref{eq:alphaderivative},
\begin{equation}
\label{eq:alphacontrol}
\frac{d}{d\alpha} \lambda_n(\Omega,\alpha) = \frac{\|u\|_{2,\partial\Omega}^2}{\|u\|_{2,\Omega}^2} \leq
	\frac{\lambda_n (\Omega, \alpha)}{\alpha}
\end{equation}
for almost every $\alpha>0$. Fixing now $\alpha_0 > 0$ and an arbitrary sequence $0 < \alpha_k \leq \alpha_0$, $k \geq 1$ with $\alpha_k \to \alpha_0$ monotonically, we may rewrite \eqref{eq:alphacontrol} as
\begin{displaymath}
\frac{d}{d\alpha} \lambda_n(\Omega,\alpha) \leq C \lambda_n (\Omega, \alpha),
\end{displaymath}
for almost all $\alpha \in [\alpha_1,\alpha_0]$, where $C=1/\alpha_1$ is independent of $\Omega$. Integrating this inequality and using the Fundamental Theorem of Calculus applied to the absolutely continuous function $\lambda_n(\Omega,\alpha)$ \cite[Theorem~7.18]{rudin},
\begin{displaymath}
	\lambda_n(\Omega,\alpha_0) \leq \lambda_n(\Omega,\alpha_k) e^{C(\alpha_0-\alpha_k)}
\end{displaymath}
for all $\alpha_k$, $k \geq 1$, and all $\Omega$. Letting $\Omega_k^*$ be an optimising domain at $\alpha_k$ for each $k \geq 1$. Then
\begin{displaymath}
\begin{split}
\lambda_n^*(V,\alpha_0) &\leq \liminf_{k \to \infty} \lambda_n(\Omega_k^*,\alpha_0)\\
	&\leq \limsup \lambda_n(\Omega_k^*,\alpha_k)e^{\alpha_0-\alpha_k}\\
	&= \lim_{k\to\infty} \lambda_n(\Omega_k^*,\alpha_k) = \lim_{k\to\infty}\lambda_n^*(V,\alpha_k).
\end{split}
\end{displaymath}
Since this holds for an arbitrary increasing sequence $\alpha_k \to \alpha_0$, and since the reverse inequality is obvious from monotonicity, this proves continuity.

Finally, that $\lambda_n^*(V,0) = 0$ follows from considering any domain with at least $n$ connected components, while to show that $\lambda_n^*(V,\alpha) < \lambda_n^*(V,\infty)$, let $\widehat\Omega$ be a domain such that $|\widehat\Omega|=V$ and $\lambda_n^*(V,\infty) = \lambda_n^D(\Omega) > \lambda_n(\Omega,\alpha) \geq \lambda_n^*(V,\alpha)$.
\end{proof}

\begin{proof}[Proof of Proposition~\ref{prop:starvst}]
As continuity and monotonicity mirror the proof of Proposition~\ref{prop:starvsalpha} closely, we do not go into great detail, but note that now left continuity is obvious and the proof of right continuity uses the property \eqref{eq:tderivative} to give us the bound
\begin{displaymath}
	\frac{d}{dt}\lambda_n(t\Omega,\alpha) \geq -\frac{2}{t^3}\lambda_n(\Omega,\alpha t)
\end{displaymath}
in place of \eqref{eq:alphacontrol}. If $t_k \to t_0$ is a decreasing sequence, then for $C=2/t_1^3$ this implies
\begin{displaymath}
	\lambda_n(t_0\Omega,\alpha) \leq \lambda_n(t_k\Omega,\alpha) e^{C(t_k-t_0)}
\end{displaymath}
for all $k\geq 1$ and all $\Omega$, from which right continuity follows in the obvious way.

We now prove that $\lambda_n^*(V,\alpha) \to \infty$ as $V \to 0$. By \eqref{eq:astscaling}, this is equivalent to $t^2 \lambda_n^*(1,\alpha/t) \to \infty$ as $t \to \infty$. Now
\begin{displaymath}
	\lambda_n^*(1,\alpha/t) \geq \lambda_1^*(1,\alpha/t) = \lambda_1(B',\alpha/t),
\end{displaymath}
where $B'$ is any ball of volume $1$. Since $\lambda_1(B',\beta)$ is concave in $\beta$ with $\lambda_1(B',0)=0$, we have $\lambda_1(B',\alpha/t)\geq \lambda_1(B',\alpha)/t$, so that
\begin{displaymath}
	t^2\lambda_n^*(1,\alpha/t) \geq t^2\lambda_1(B',\alpha/t) \geq t\lambda_1(B',\alpha) \longrightarrow \infty
\end{displaymath}
as $t \to \infty$. Finally, to show that $\lambda_n^*(V,\alpha) \to 0$ as $V \to \infty$, or equivalently, $t^2 \lambda_n^*(1, \alpha/t) \to 0$ as $t \to 0$, we simply note that
\begin{displaymath}
	t^2\lambda_n^*(1,\alpha/t) \leq t^2\lambda_n^*(1,\infty) = C(N,n)t^2 \longrightarrow 0
\end{displaymath}
as $t \to 0$, where $\lambda_n^*(1,\infty)$ is, as usual, the corresponding infimum for the Dirichlet problem.
\end{proof}

\section{General description of the numerical optimisation procedure\label{genopt}}

The numerical solution of the shape optimisation problem is divided in two steps. At a first level we will describe the application of the Method of Fundamental Solutions (MFS) to the calculation of Robin eigenvalues for a fixed domain. Then, we will use a steepest descent method (eg.~\cite{nocedal}) to determine optimal domains for each of the Robin eigenvalues.

\subsection{Numerical calculation of Robin eigenvalues using the MFS}

We will consider the numerical optimisation of Robin eigenvalues in the class of sets $\Xi$ which are built with a finite number of star shaped and bounded planar domains. For simplicity, for now, assume that $\Omega\in\Xi$ has only one connected component. Thus, $\Omega$ is isometric to a domain $\Omega_\infty$ defined in polar coordinates by
\[\Omega_\infty=\left\{(r,\theta):0<r<r_\infty(t)\right\}\]
where
\[r_\infty(t)=a_0+\sum_{i=1}^{\infty}\left(a_i\cos(i\theta)+b_i\sin(i\theta)\right).\]
Now we consider the approximation
\begin{equation}
\label{dominio}
r_\infty(t)\approx r_M(t):= a_0+\sum_{i=1}^{M}\left(a_i\cos(i\theta)+b_i\sin(i\theta)\right),
\end{equation}
for a given $M\in\mathbb{N}$ and the approximated domain $\Omega_M$ defined by
\begin{equation}
\label{dominioaprox}
\Omega_M=\left\{(r,\theta):0<r<r_M(t)\right\}.
\end{equation}
Now we describe how to apply the MFS for the calculation of Robin eigenvalues of $\Omega_M$. We take the
fundamental solution of the Helmholtz equation
\begin{equation}
\Phi_{\lambda}(x)=\frac{i}{4}H_{0}^{(1)}(\sqrt{\lambda}\left\Vert
x\right\Vert )
\end{equation}
where $\left\Vert.\right\Vert $ denotes the Euclidean norm in $\mathbb{R}^{2}$ and $H_{0}^{(1)}$ is the first H\"{a}nkel function. An eigenfunction solving the eigenvalue problem \eqref{eq:robin} is approximated by a linear combination
\begin{equation}
\label{mfsapprox}
u(x)\approx\tilde{u}(x)=\sum_{j=1}^{N_p}\beta_j\phi_j(x),
\end{equation}
where
\begin{equation}
\label{ps} \phi_{j}(x)=\Phi_{\lambda}(x-y_{j})
\end{equation}
are $N_p$ point sources centered at some points $y_{j}$ placed on an admissible source set which does not intersect
$\bar{\Omega}$ (eg.~\cite{alan}). Each of the point sources $\phi_{j}$ satisfies the partial differential equation of the eigenvalue problem and thus, by construction, the MFS approximation also satisfies the partial differential equation of the problem. We take $N_p$ collocation points $x_i,\ i=1,...,N_p$ almost equally spaced on $\partial\Omega_M$ and for each of these points we determine the outward unitary vector $n_i$ which is normal to the boundary at $x_i$. The source points $y_i$ are calculated by
\begin{equation}
\label{pontosfonte}
y_i=x_i+\gamma\ n_i,\ i=1,...,N_p
\end{equation}
where $\gamma$ is a constant (see~\cite{alan} for details). This choice of collocation and source points is illustrated in Figure~\ref{fig:pontosfonte} where the collocation points on the boundary and the source points are shown in black and grey, respectively.
\begin{figure}[!htb]\centering
  \includegraphics[width=0.48\textwidth]{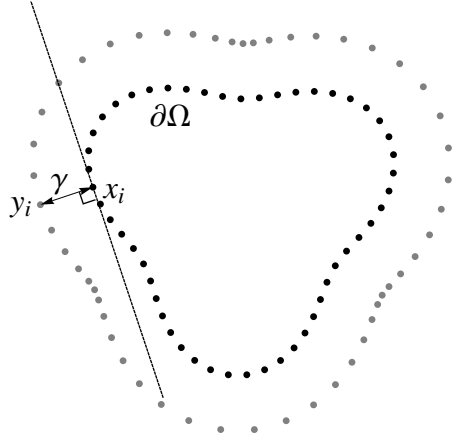} %
  \caption{The collocation and source points.}
  \label{fig:pontosfonte}
\end{figure}
The coefficients in the MFS approximation \eqref{mfsapprox} are determined by imposing the boundary conditions of the problem at the collocation points,
\[\frac{\partial \tilde{u}}{\partial\nu}(x_i)+\alpha \tilde{u}(x_i)=0,\ i=1,...,N_p.\]
This leads to a linear system
\begin{equation}
\label{sistema}A(\lambda).\overrightarrow{\beta}=\overrightarrow{0},
\end{equation}
where
\[A_{i,j}(\lambda)=\frac{\partial \phi_{j}}{\partial\nu}(x_i)+\alpha \phi_{j}(x_i),\ i,j=1,...,N_p,\]
writing $\overrightarrow{\beta}=\left[\beta_1\ ...\ \beta_{N_p}\right]^T$ and $\overrightarrow{0}=\left[0\ ...\ 0\right]^T.$ As in~\cite{alan}, we calculate the approximations for the eigenvalues by determining the values of $\lambda$ such that the system \eqref{sistema} has non trivial solutions.

\subsection{Numerical shape optimisation}

Now we will describe the algorithm for the numerical solution of shape optimisation problem associated to the Robin eigenvalues. Each vector $\mathcal{C}:=(a_{0},a_{1},...,a_{M},b_{1},b_{2},...,b_{M})$ defines a domain using \eqref{dominio} and \eqref{dominioaprox}. The optimisation problems are solved by considering each Robin eigenvalue as a function of $\mathcal{C}$, $\lambda_n(\mathcal{C})$, and determining optimal vectors $\mathcal{C}$. Now we note that we must take into account the area constraint in problem \eqref{eq:ninf}. For each vector $\mathcal{C}$ we define a corresponding normalised vector $\widehat{\mathcal{C}}$ given by
\[\widehat{\mathcal{C}}=\frac{\mathcal{C}}{\left|\Omega_M\right|^{\frac{1}{2}}}.\]
The domain associated to $\widehat{\mathcal{C}}$ has unit area. Now, for each component of the vector $\mathcal{C}$, we define an approximation for the derivative of a Robin eigenvalue with respect to this component, given simply by a finite difference
\[\lambda_{n,i}'=\frac{\lambda_n\left(\widehat{\mathcal{P}_i}\right)-\lambda_n\left(\widehat{\mathcal{C}}\right)}{\varepsilon},\]
for a small value $\varepsilon$, where $\mathcal{P}_i=\widehat{\mathcal{C}}+\varepsilon\ e_i,\ i=1,...,2M+1$ and
$e_1=[1\ 0\ ...\ 0]^T$, $e_2=[0\ 1\ 0\ 0\ ...\ 0]^T$, $e_3=[0\ 0\ 1\ 0\ ...\ 0]^T$, \ldots. We then build the approximation of the gradient
\[d_n=\left[\lambda_{n,1}'\ \lambda_{n,2}'\ ...\ \lambda_{n,2M+1}'\right]^T\]
which defines the searching direction for the steepest descent method. We start by defining $\mathcal{C}_0=\widehat{\mathcal{C}}$ and calculate the next points $\mathcal{C}_{n+1}$ solving the minimisation problem
\[\textrm{Min}_{x}\ \lambda_n\left(\widehat{\mathcal{C}_n-x d_n}\right)\]
using the golden ratio search (eg.~\cite{alan}). Once we have calculated the optimal step length $\delta$ we define
\[\mathcal{C}_{n+1}=\widehat{\mathcal{C}_n-\delta d_n},\ n=0,1,\ldots.\]
An alternative procedure would be to apply an optimisation numerical method similar to that studied in~\cite{curtovert}, considering the area constraint.

\subsubsection{Multiple eigenvalues in the optimisation process}

As in~\cite{anfr}, in the optimisation process we must deal with multiple eigenvalues. For each $n$, we start minimising $\lambda_n\left(\widehat{\mathcal{C}}\right)$ and once we obtain
\[\lambda_n\left(\widehat{\mathcal{C}}\right)-\lambda_{n-1}\left(\widehat{\mathcal{C}}\right)\leq\theta,\]
for a small value $\theta<1$, we modify the function to be minimised and try to minimise
\[\lambda_n\left(\widehat{\mathcal{C}}\right)-\omega_{n-1}\log\left(\lambda_n\left(\widehat{\mathcal{C}}\right)-\lambda_{n-1}\left(\widehat{\mathcal{C}}\right)\right)\]
for a sequence of constants $\omega_{n-1}\searrow0$. Then, once we obtain
\[\lambda_n\left(\widehat{\mathcal{C}}\right)-\lambda_{n-1}\left(\widehat{\mathcal{C}}\right)\leq\theta\ \text{ and }\ \lambda_{n-1}\left(\widehat{\mathcal{C}}\right)-\lambda_{n-2}\left(\widehat{\mathcal{C}}\right)\leq\theta,\]
we change the function to be minimised to
\[\lambda_n\left(\widehat{\mathcal{C}}\right)-\omega_{n-1}\log\left(\lambda_n\left(\widehat{\mathcal{C}}\right)-\lambda_{n-1}\left(\widehat{\mathcal{C}}\right)\right)-\omega_{n-2}\log\left(\lambda_{n-1}\left(\widehat{\mathcal{C}}\right)-\lambda_{n-2}\left(\widehat{\mathcal{C}}\right)\right)\]
for suitable choice of constants $\omega_{n-1},\omega_{n-2}\searrow0$ and continue applying this process, adding more eigenvalues to the linear combination defining the function to be minimised until we find the optimiser and the multiplicity of the corresponding eigenvalue.

\subsubsection{The case of disconnected domains}

In the case for which we have a set consisting of several connected components, we simply consider a vector $\mathcal{C}$ defining each component, such that the sum of the areas is equal to one and then perform optimisation on these vectors as described above. The application of the MFS in this case is straightforward, considering collocation points uniformly distributed on the boundary of each of the components and for each of these collocation points, calculate a source point by \eqref{pontosfonte}. Note that the parametrisation of domains that we considered limits the possible shapes to finite unions of star-shaped components. We know that each domain $\Omega_k^\ast$ has at most $k$ connected components (see Remark~\ref{rem:existence}) and thus, in the shape optimisation of $\lambda_k$, there is only a finite number of possibilities for building an optimal disconnected domain. These
were studied exhaustively, working with a fixed number of connected components at each step and using the Wolf--Keller type result above (Theorem~\ref{th:wkrobin}) to test the numerical results obtained. This process becomes more difficult for higher eigenvalues and in that case it might be preferable to use a level set method, as in~\cite{oude}, for instance.

\section{Numerical results}

In this section we will present the main results obtained from our numerical study on the minimisation of the first seven Robin eigenvalues for domains with unit area. We will write $B_n$ as shorthand for the domain of unit area composed of $n$ equal balls. It is well known that the first and second Robin eigenvalues are minimised by $B_1$ and $B_2$, respectively. Figure~\ref{fig:lambda12} shows the evolution of the optimal values of $\lambda_1$ and $\lambda_2$, as a function of the Robin parameter $\alpha$.
\begin{figure}[!htb]\centering
  \includegraphics[width=0.6\textwidth]{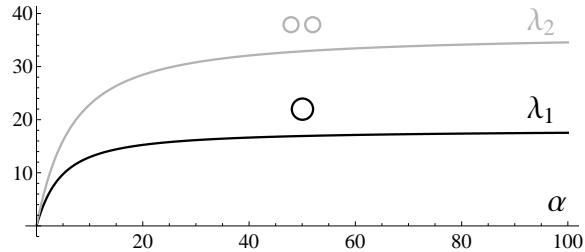} %
  \caption{Robin optimisers for $\lambda_1$ and $\lambda_2$.}
  \label{fig:lambda12}
\end{figure}

In Figure~\ref{fig:lambda3} we plot the optimal value of $\lambda_3$. We can observe that there are two types of optimal domains depending on the value of $\alpha$. More precisely, for $\alpha\leq\alpha_3\approx14.51236$, the third Robin eigenvalue is minimised by $B_3$, while for $\alpha\geq\alpha_3$, the ball $B_1$ seems to be the minimiser. In particular, for $\alpha=\alpha_3$ uniqueness of the minimiser appears to fail, the optimal value of $\lambda_3$ being attained by both domains. Note also that in the asymptotic case when $\alpha\rightarrow\infty$ this result agrees with the conjecture that the ball is the Dirichlet minimiser of the third eigenvalue (cf.~\cite{wolf,henr}).
\begin{figure}[!htb]\centering
  \includegraphics[width=0.6\textwidth]{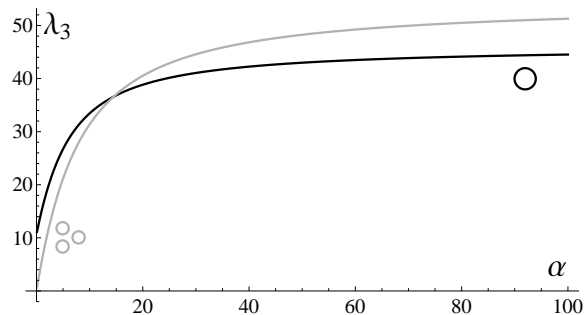} %
  \caption{Robin optimisers for $\lambda_3$.}
  \label{fig:lambda3}
\end{figure}

While for $\lambda_3$, the Dirichlet minimiser is also the Robin minimiser for $\alpha\geq\alpha_3$, the situation is different for higher eigenvalues. In Figure~\ref{fig:lambda45}-left we plot results for the minimisation of $\lambda_4$. We have marked with a dashed line the curve associated with the fourth Robin eigenvalue of the Dirichlet minimiser which is conjectured to be the union of two balls whose radii are in the ratio $\sqrt{j_{0,1}/j_{1,1}}$, where $j_{0,1}$ and $j_{1,1}$ are respectively the first zeros of the Bessel functions $J_0$ and $J_1$ (eg.~\cite{henr}). We can observe, as was to be expected, that it is not optimal. The solid curve below it also corresponds to domains built with two balls with different areas, but whose optimal ratio of the areas depends on the Robin parameter $\alpha$. It is this family of domains which appears to be minimal for larger $\alpha$. Again we have a value $\alpha_4\approx16.75743$ for which $\alpha\leq\alpha_4$ implies that the minimiser is $B_4$. In Figure~\ref{fig:lambda4areas} we plot the area of the largest ball in the optimal set $\Omega_4^\ast$, as a function of $\alpha\in[\alpha_4,100]$. We marked with a dashed line the asymptotic case of the Dirichlet optimiser for which this quantity is equal to $\frac{j_{1,1}^2}{j_{0,1}^2+j_{1,1}^2}\approx0.7174.$
\begin{figure}[!htb]\centering
  \includegraphics[width=0.55\textwidth]{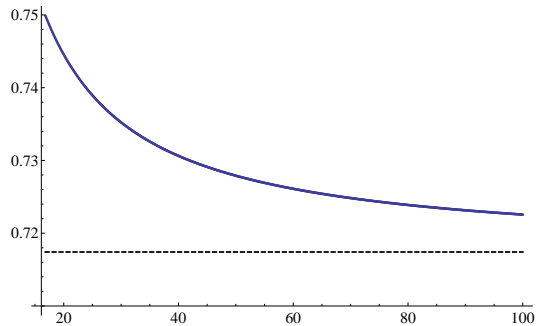} %
    \caption{Area of the largest ball of $\Omega_4^\ast$.}
  \label{fig:lambda4areas}
\end{figure}
In Figure~\ref{fig:lambda45}-right we show results for the minimisation of $\lambda_5$. The curve corresponding to the Dirichlet minimiser found in~\cite{anfr} is again marked with a dashed line and the dotted line below it represents a family of domains of  a very similar shape, deforming slowly, which appear to be optimisers at their respective values of $\alpha\geq\alpha^\ast\approx40$.
\begin{figure}[!htb]\centering
  \includegraphics[width=0.49\textwidth]{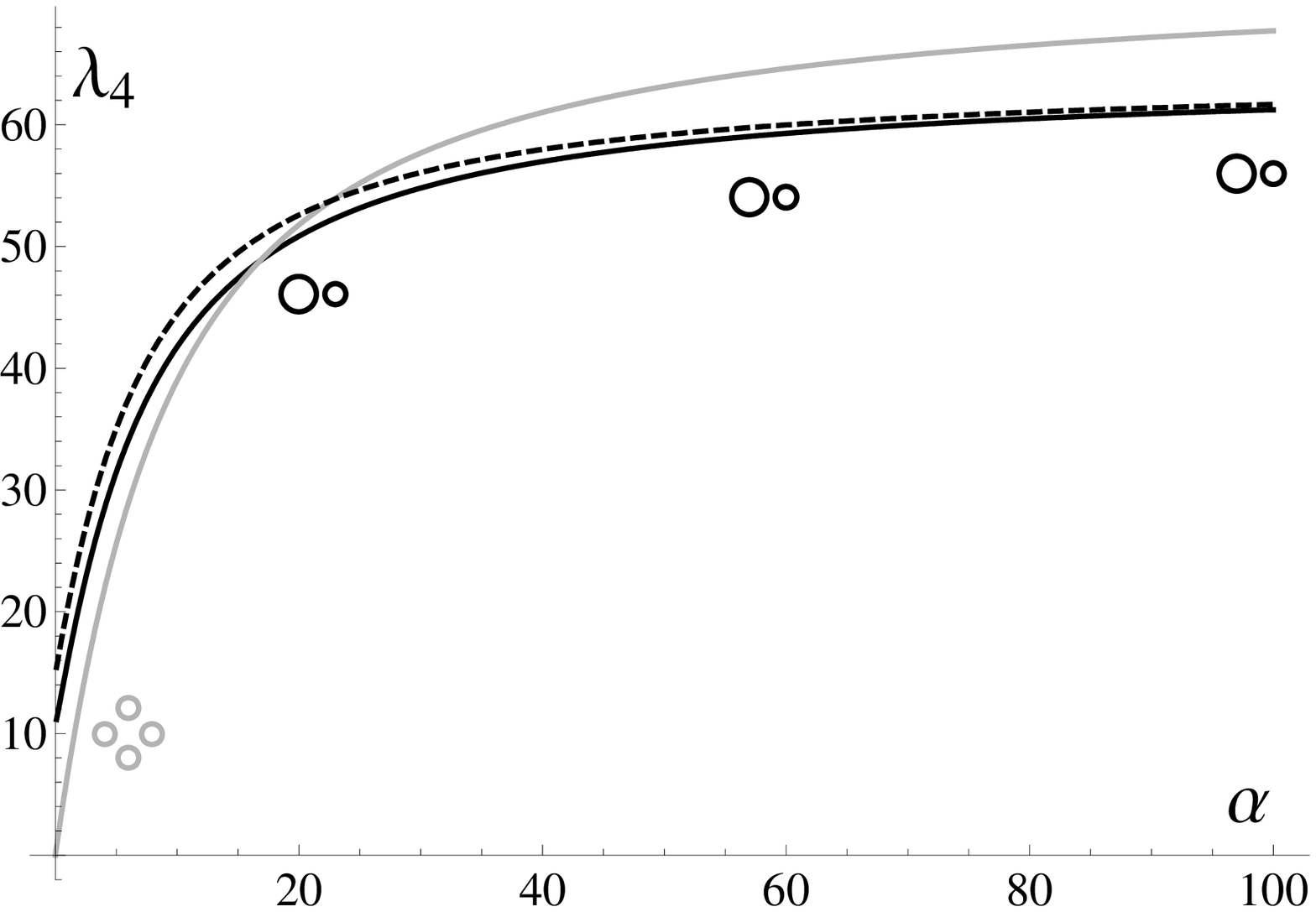} %
  \includegraphics[width=0.49\textwidth]{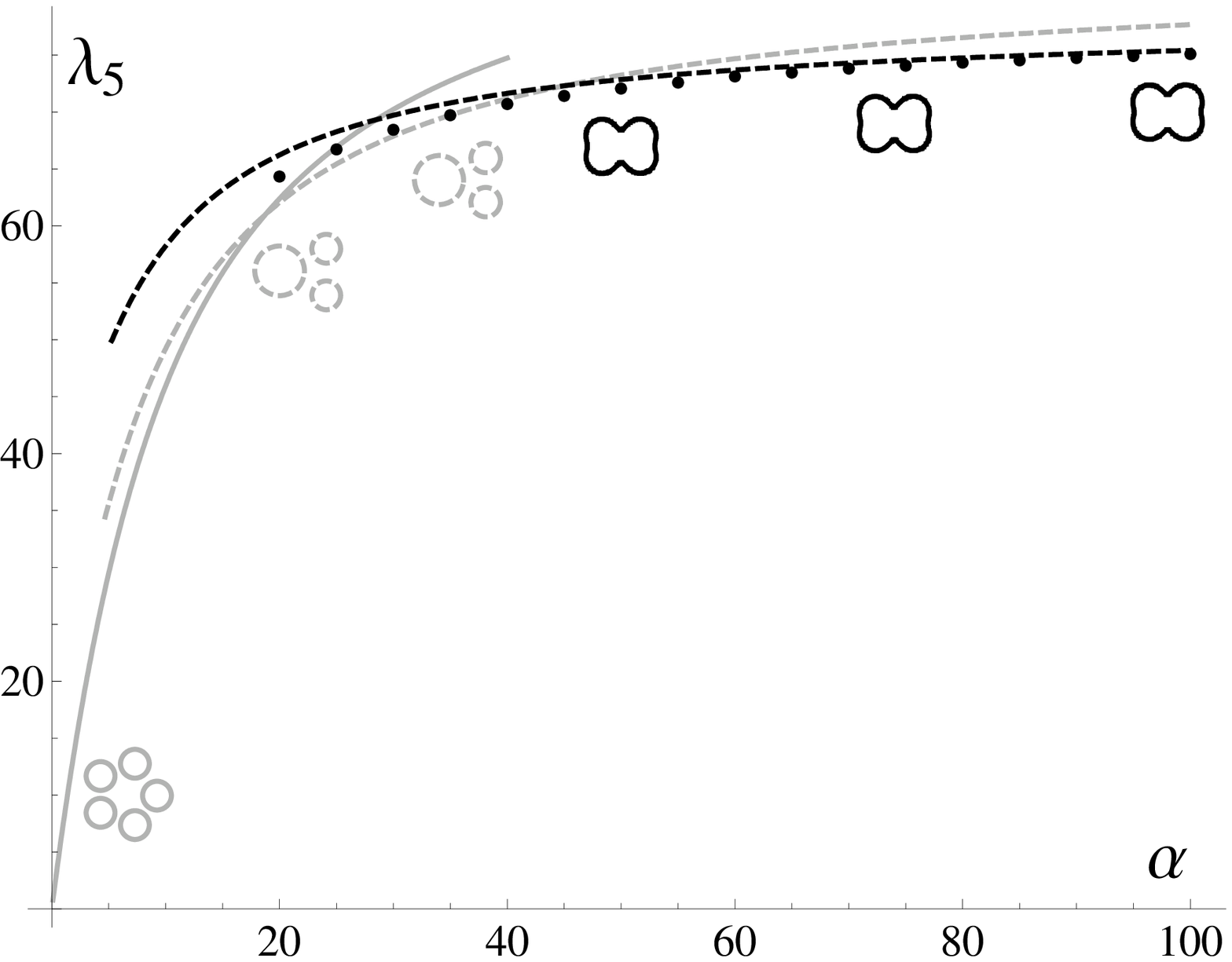} %
  \caption{Robin optimisers of $\lambda_4$ and $\lambda_5$.}
  \label{fig:lambda45}
\end{figure}
Then, for $\alpha_5\leq\alpha\leq\alpha^\ast$, where $\alpha_5\approx18.73537$, the minimiser is a set composed by a big ball and two small balls of the same area, which corresponds to a union of scaled copies of $B_1$ and $B_2$, and where again the optimal ratio of these two areas depends on $\alpha$. For $\alpha\leq\alpha_5$, $\lambda_5$ is minimised by $B_5$.

In Figure~\ref{fig:lambda6}-left we plot the results for the minimisation of $\lambda_6$. The curve corresponding to the Dirichlet minimiser is again marked with a dashed line, while the Robin optimisers for large $\alpha$, again close to their Dirichlet counterpart, are marked as a sequence of points below it. For some smaller values of $\alpha$ it appears that two balls of the same area are the minimiser and for $\alpha\leq\alpha_6\approx20.52358$, $\lambda_6$ is minimised by $B_6$. Figure~\ref{fig:lambda6}-right shows a zoom of the previous figure in the region obtained for $\alpha\in[10,40]$.
\begin{figure}[!htb]\centering
 \includegraphics[width=0.49\textwidth]{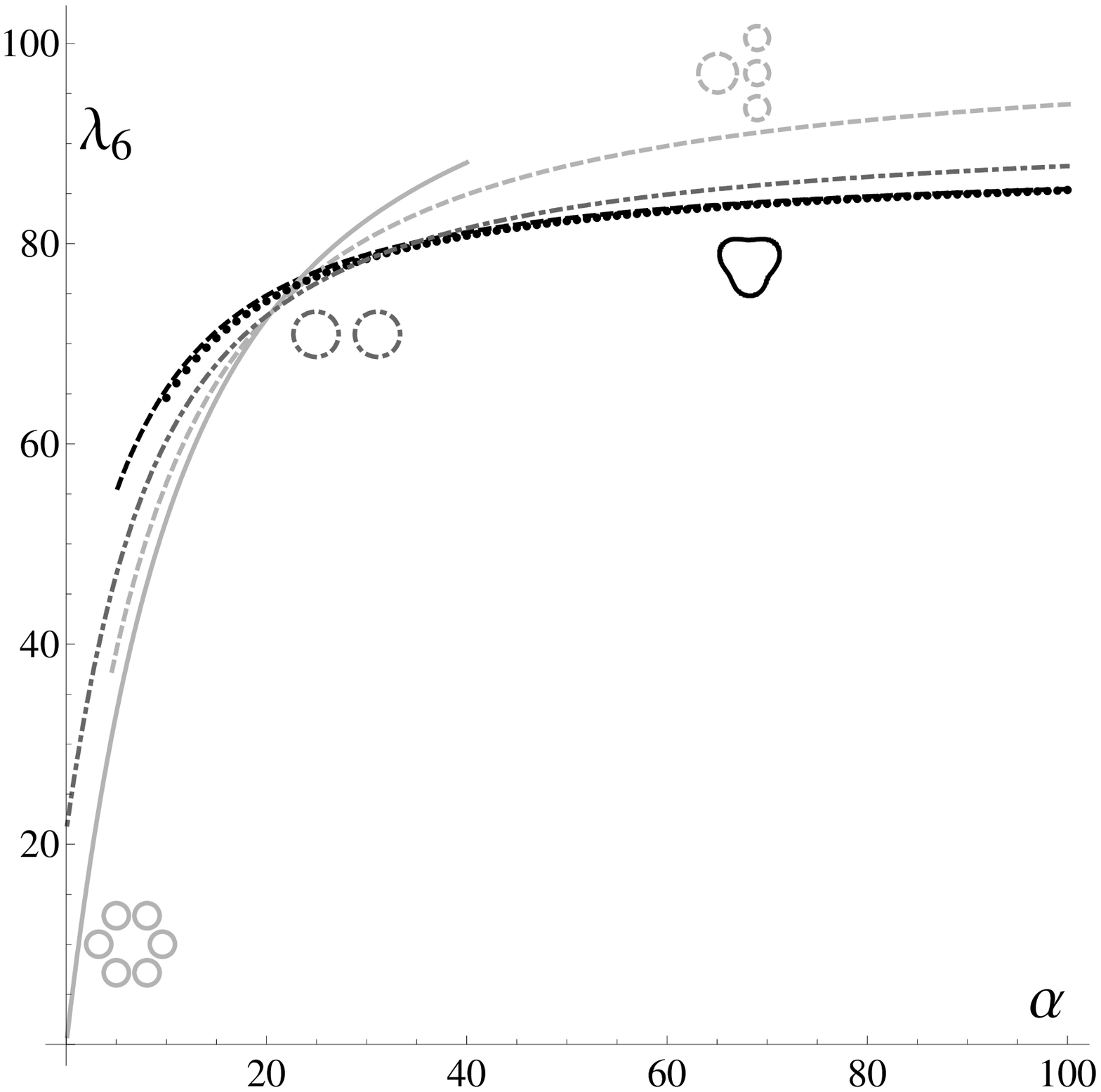} %
 \includegraphics[width=0.49\textwidth]{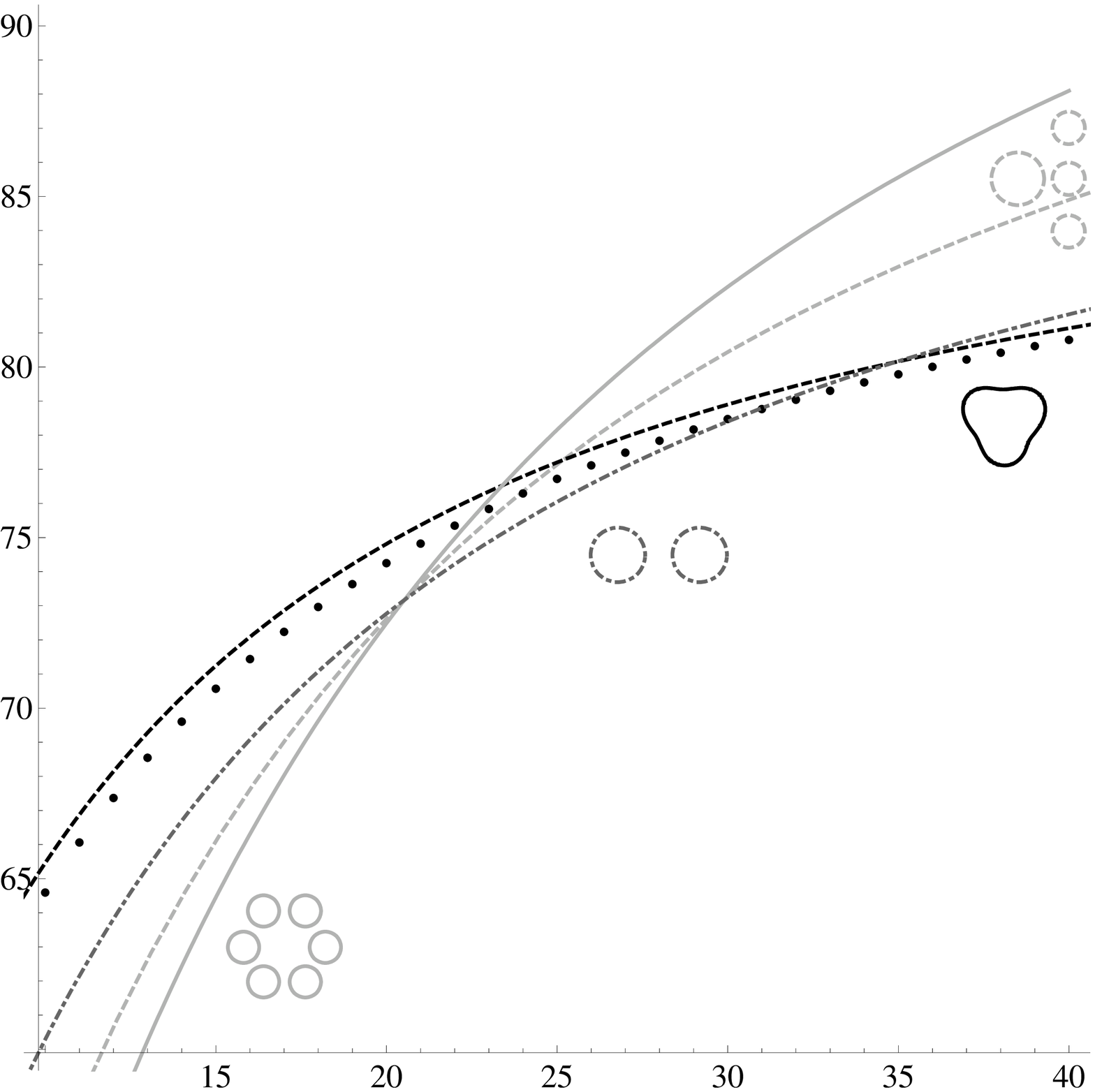} %
 \caption{Robin optimisers for $\lambda_6$ and a zoom of the region associated with $\alpha\in[10,40]$.}
 \label{fig:lambda6}
\end{figure}
We can observe that there is a particular value of $\alpha$ for which the three curves associated to unions of balls have an intersection. For this particular $\alpha$ we have three distinct optimisers.

Finally, in Figure~\ref{fig:lambda7} we show the results for the minimisation of $\lambda_7$. Again the dashed curve is associated to the Dirichlet minimiser and the points below it the optimal values obtained for some domains which are very similar to the Dirichlet optimiser.
\begin{figure}[!htb]\centering
  \includegraphics[width=0.6\textwidth]{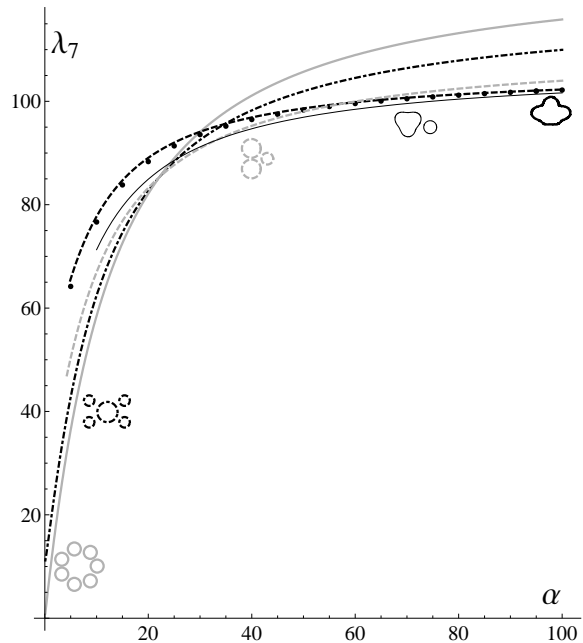} %
  \caption{Robin optimisers for $\lambda_7$.}
  \label{fig:lambda7}
\end{figure}

It is interesting to note that for the region associated with $\alpha\leq100$ which is plotted in the figure we have a family of domains whose curve is below the curve associated to the type of domains which are very similar to the Dirichlet optimiser. In the spirit of our Wolf-Keller type theorem (Theorem~\ref{th:wkrobin}), this curve corresponds to sets built with an optimiser of $\lambda_6$ and a small ball, which is an optimiser for $\lambda_1$. Since the Dirichlet optimiser is connected (cf.~\cite{anfr}), we expect that for some $\alpha>100$ the value obtained for this type of domain will become larger than the value obtained for domains similar to the Dirichlet minimiser, and indeed, testing the algorithm for $\alpha=300$, we find that the optimal domain is connected. For $\alpha\leq\alpha_7\approx22.167800$, the minimiser is $B_7$, while for $\alpha=\alpha_7$ we have three different types of optimal unions of balls.

\subsection{Optimal unions of balls}

The above results suggest that the minimiser of $\lambda_{n}$ in two dimensions and for small $\alpha$ consists of the set $B_n$ formed by $n$ equal balls, and that this situation changes when the line corresponding to $B_n$ intersects the line corresponding to $n-3$ equal balls
with the same radius and a larger ball, that is, a union of scaled copies of $B_{n-3}$ and $B_1$. The following lemma gives the value at which this intersection takes place in terms of the root of an equation involving Bessel functions, showing that, for dimension $N$, this value grows with $n^{1/N}$.

\begin{lemma}\label{intersection}
Consider the Robin problem for domains in $\mathbb R^{N}$ of volume $V$. The value $\alpha_{n}$ for which the $n^{\rm th}$ eigenvalue of $n$ equal balls equals the $n^{\rm th}$ eigenvalue of the set formed by $n-(N+1)$ equal balls with the same radius as before and a larger ball is
given by
\[
\alpha_{n} = \gamma_{0}\frac{\ds J_{\frac{N}{2}}(\gamma_{0})}{\ds J_{\frac{N}{2}-1}(\gamma_{0})}
\left(\frac{\ds n\omega_{N}}{\ds V}\right)^{1/N},
\]
where $\gamma_{0}$ is the smallest positive solution of the equation
\begin{equation}\label{eta0}
J_{\frac{N}{2}-1}(\gamma)J_{\frac{N}{2}}(C_{N}\gamma)-C_{N}\gamma J_{\frac{N}{2}-1}(\gamma)J_{\frac{N}{2}+1}(C_{N}\gamma)
+ C_{N}\gamma J_{\frac{N}{2}}(\gamma)J_{\frac{N}{2}}(C_{N}\gamma)=0,
\end{equation}
with $C_{N} = (N+1)^{1/N}$.
\end{lemma}

\begin{proof}
The first eigenvalue of $B_{n}$, say $\lambda$, has multiplicity $n$ and is given by the smallest positive solution of the equation
\[
\sqrt{\lambda}J_{\frac{N}{2}}\Big(\sqrt{\lambda}r_{1}\Big)-\alpha J_{\frac{N}{2}-1}(\sqrt{\lambda}r_{1}) =0,
\]
where $r_{1}$ is the radius of each ball, given by $(V/(n \omega_{N}))^{1/N}$. If the smaller balls
of the domain formed by $n-(N+1)$ equal balls and a larger ball all have radius $r_{1}$, then the larger ball will have volume $(N+1)V/n$ and radius $r_{2}=C_{N}r_{1}$. We thus have that $\lambda$ equals the $n^{\rm th}$ eigenvalue of this second domain, provided that the second eigenvalue of the larger ball,
which has multiplicity $N$, also equals $\lambda$. This is now given by the smallest positive solution of the equation
\[
(1+\alpha r_{2})J_{\frac{N}{2}}(\sqrt{\lambda}r_{2})-
r_{2}\sqrt{\lambda}J_{\frac{N}{2}+1}(\sqrt{\lambda}r_{2})=0.
\]
Writing $\gamma=\sqrt{\lambda}r_{1}$, we see that we want to find $\gamma$ which is a solution of
\[
\left\{
\begin{array}{l}
(1+\frac{\ds \alpha}{\ds \sqrt{\lambda}}C_{N}\gamma)J_{\frac{N}{2}}(C_{N}\gamma)-C_{N}\gamma J_{\frac{N}{2}+1}(C_{N}\gamma)=0
\eqskip
\sqrt{\lambda}J_{\frac{N}{2}}(\gamma)-\alpha J_{\frac{N}{2}-1}(\gamma)=0.
\end{array}
\right.
\]
Solving with respect to $\alpha/\sqrt{\lambda}$ in the second of these equations, and replacing the expression
obtained in the first equation yields the desired result.
\end{proof}

\begin{table}
\caption{Optimisers in dimension $2$ for $\alpha=\alpha_n, \;\; n=3,...,10$.}
\begin{center}
\begin{tabular}{|c|c|llll|}
\hline n  & $\alpha_n$ &  &  & & \\
\hline 3  &  14.51236 & \includegraphics[width=0.5cm]{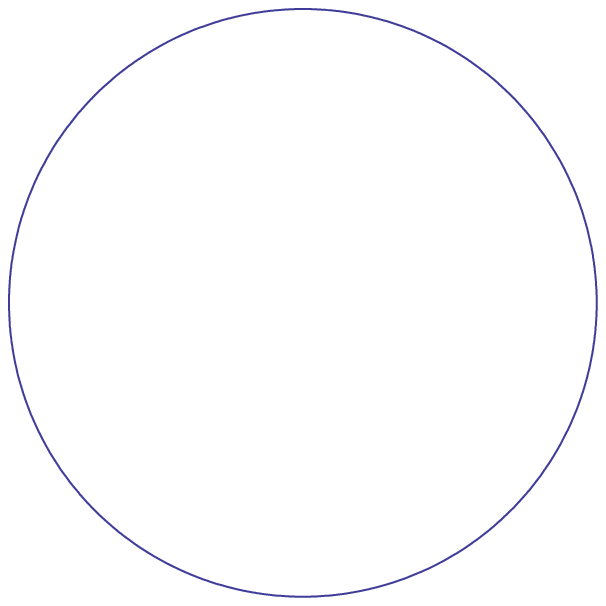} & \includegraphics[width=0.5cm]{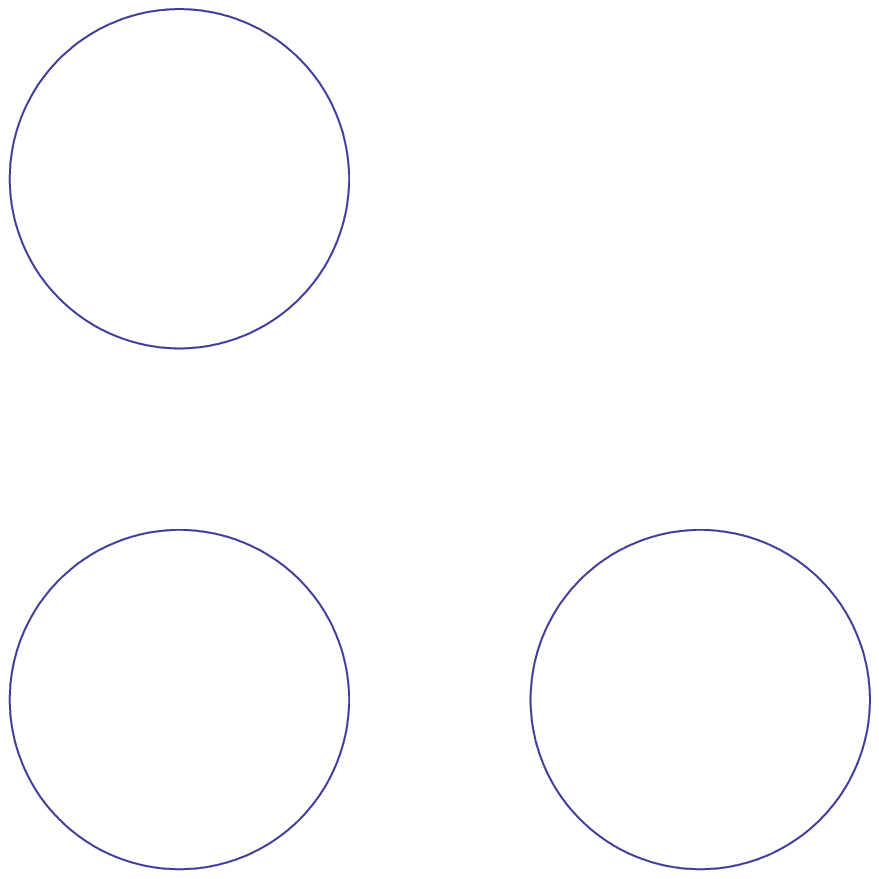} & &     \\
\hline 4  &  16.75743 & \includegraphics[width=0.433cm]{B1}\includegraphics[width=0.25cm]{B1} &
\includegraphics[width=0.5cm]{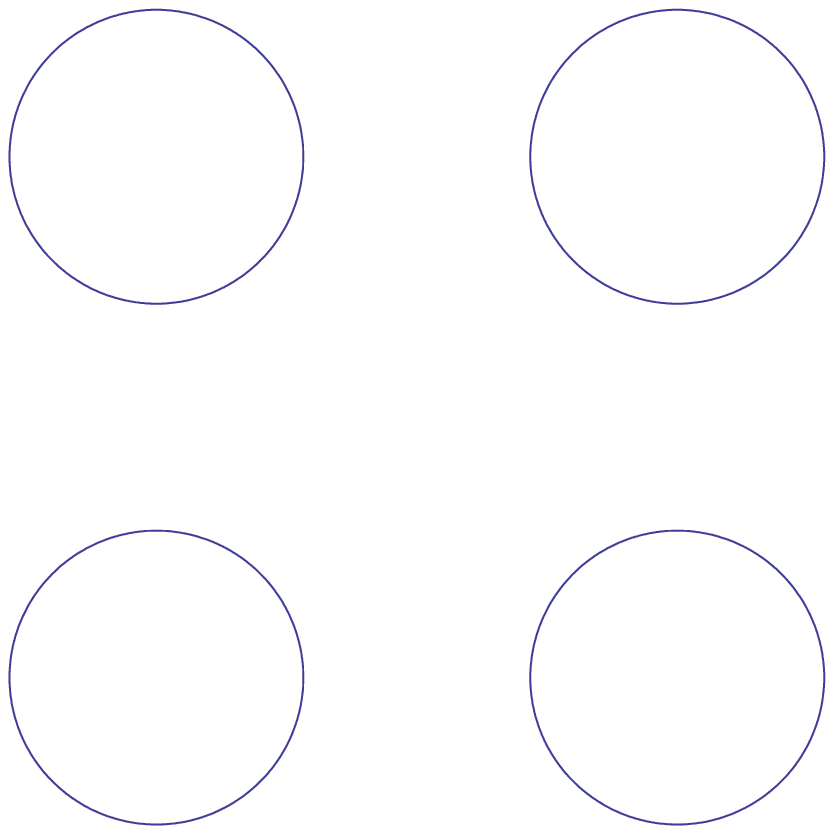} & &    \\
\hline 5  &  18.73537 & \includegraphics[width=0.387cm]{B1}\includegraphics[width=0.224cm]{B1}\includegraphics[width=0.224cm]{B1} & \includegraphics[width=0.5cm]{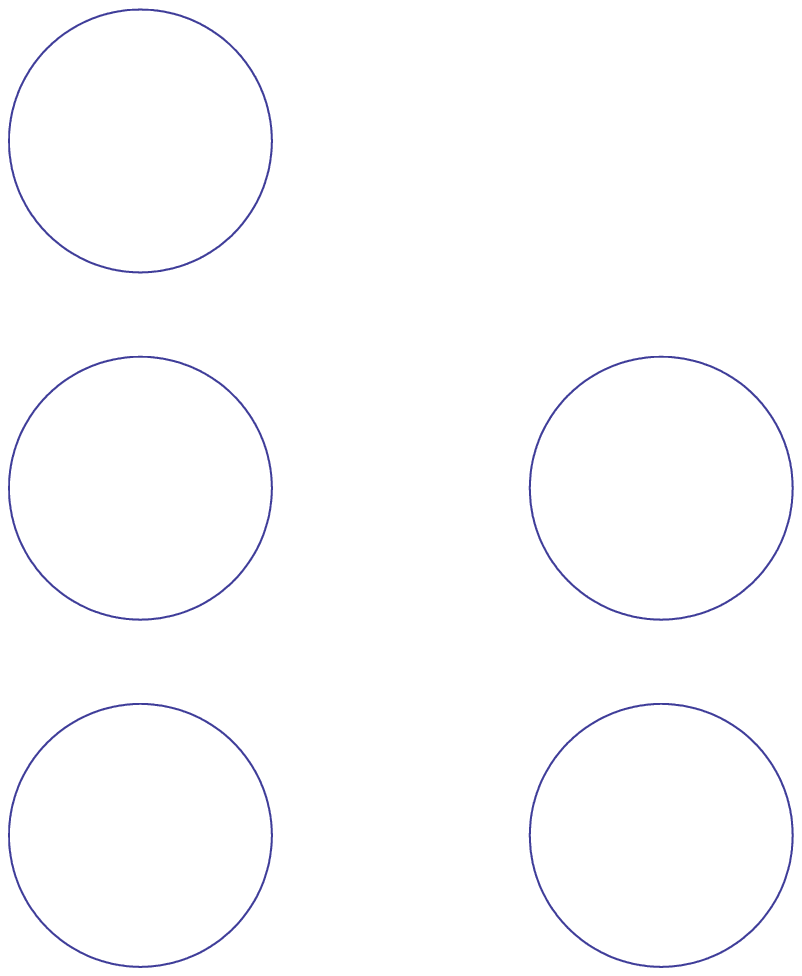} & &   \\
\hline 6  &  20.52358 & \includegraphics[width=0.354cm]{B1}\includegraphics[width=0.354cm]{B1} & \includegraphics[width=0.354cm]{B1}\includegraphics[width=0.354cm]{b3} &\includegraphics[width=0.5cm]{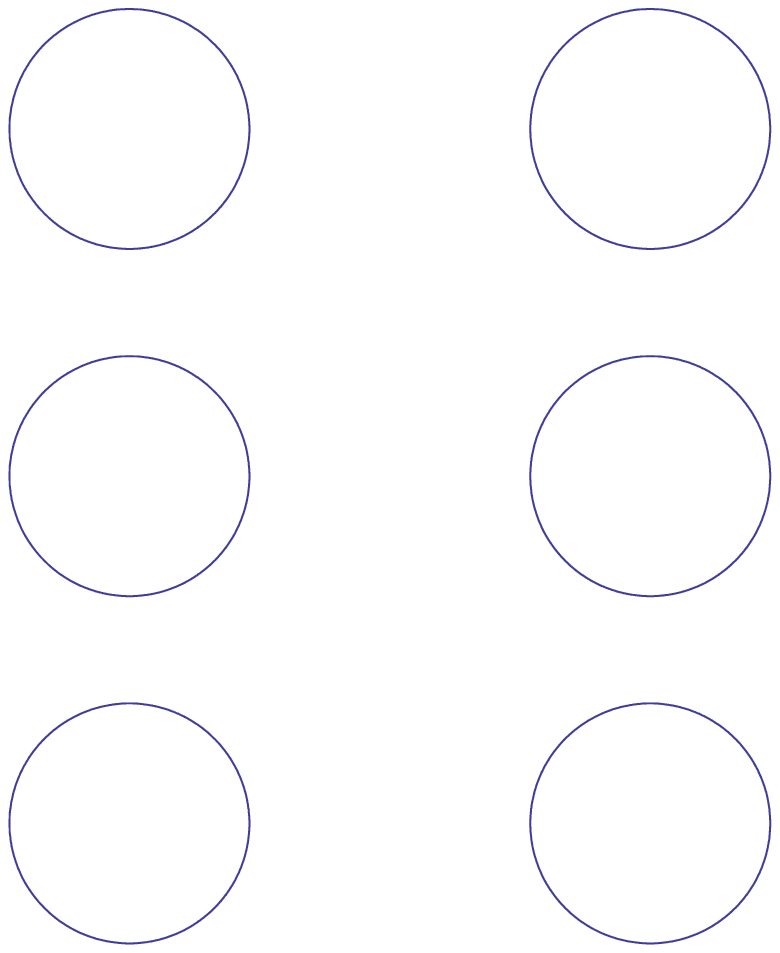} &   \\
\hline 7  &  22.16800 & \includegraphics[width=0.327cm]{B1}\includegraphics[width=0.327cm]{B1}\includegraphics[width=0.189cm]{B1} & \includegraphics[width=0.327cm]{B1}\includegraphics[width=0.378cm]{b4} & \includegraphics[width=0.5cm]{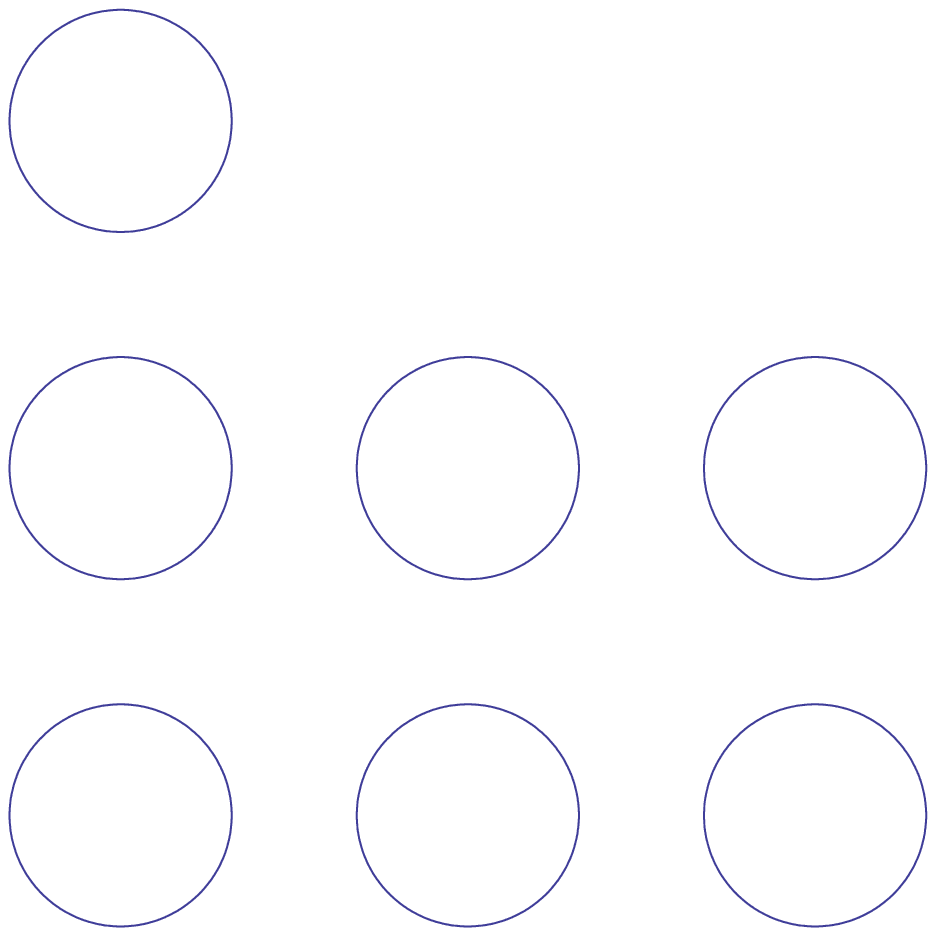} &   \\
\hline 8  &  23.69859 & \includegraphics[width=0.306cm]{B1}\includegraphics[width=0.306cm]{B1}\includegraphics[width=0.177cm]{B1}\includegraphics[width=0.177cm]{B1} & \includegraphics[width=0.306cm]{B1}\includegraphics[width=0.423cm]{b5} &\includegraphics[width=0.5cm]{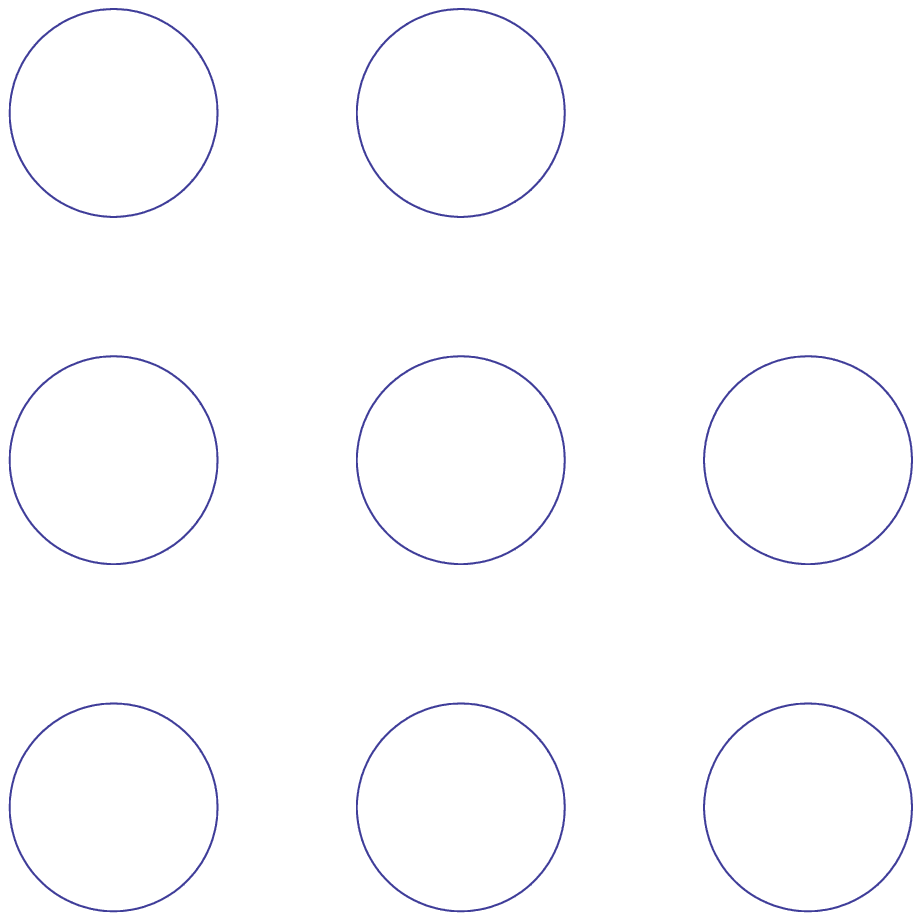} &   \\
\hline 9  &  25.13615 & \includegraphics[width=0.289cm]{B1}\includegraphics[width=0.289cm]{B1}\includegraphics[width=0.289cm]{B1} & \includegraphics[width=0.289cm]{B1}\includegraphics[width=0.289cm]{B1}\includegraphics[width=0.289cm]{b3} & \includegraphics[width=0.289cm]{B1}\includegraphics[width=0.408cm]{b6} &\includegraphics[width=0.5cm]{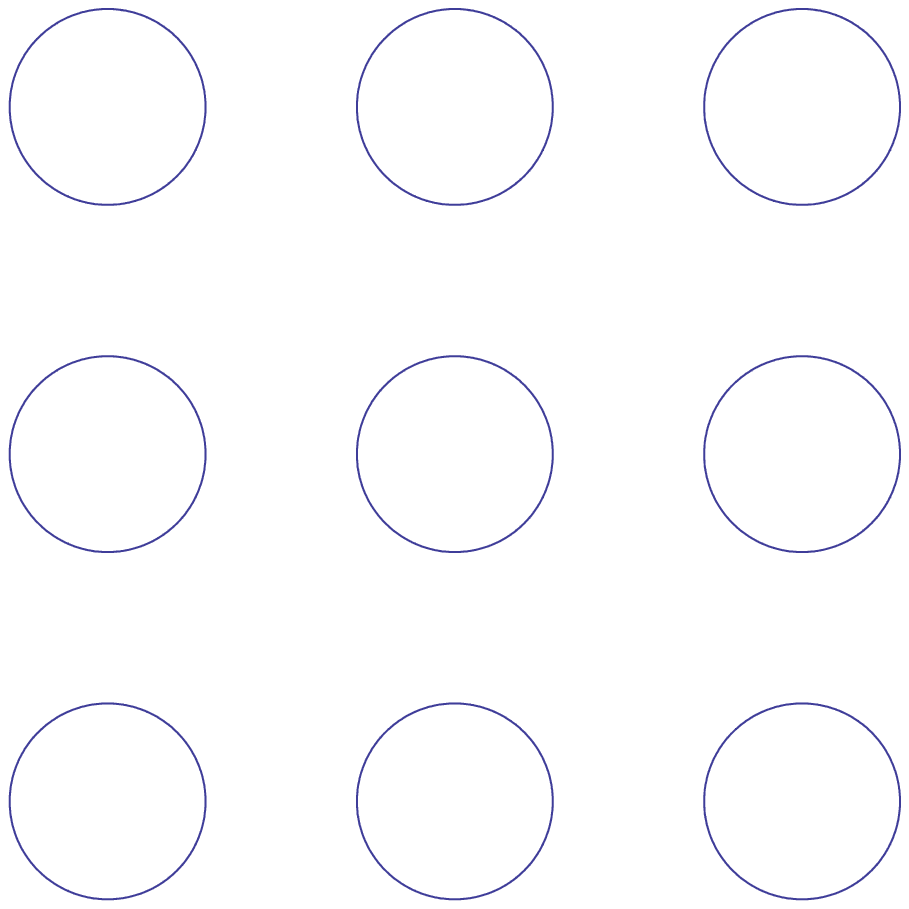}   \\
\hline 10  & 26.49583 & \includegraphics[width=0.274cm]{B1}\includegraphics[width=0.274cm]{B1}\includegraphics[width=0.274cm]{B1}\includegraphics[width=0.158cm]{B1} & \includegraphics[width=0.274cm]{B1}\includegraphics[width=0.274cm]{B1}\includegraphics[width=0.316cm]{b4} & \includegraphics[width=0.274cm]{B1}\includegraphics[width=0.418cm]{b7} & \includegraphics[width=0.5cm]{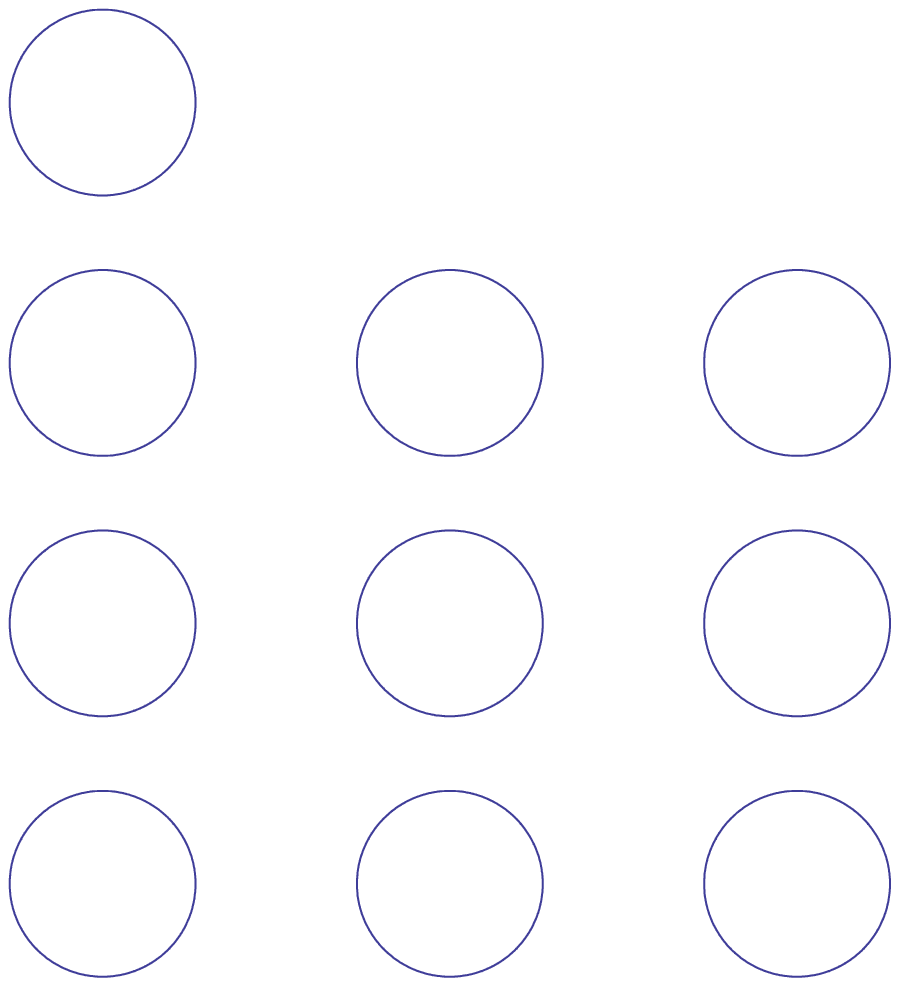}  \\
\hline
\end{tabular}
\end{center}
\label{tab1}
\end{table}

In the case of dimension two, equation~(\ref{eta0}) reduces to
\[
%J_{0}(\gamma)\left[J_{0}(\sqrt{3}\gamma)-J_{2}(\sqrt{3}\gamma)\right]+2J_{1}(\sqrt{3}\gamma)J_{1}(\gamma)=0
J_{0}(\gamma)J_{1}(\sqrt{3}\gamma)-\sqrt{3}\gamma J_{0}(\gamma)J_{2}(\sqrt{3}\gamma)+
\sqrt{3}\gamma J_{1}(\gamma)J_{1}(\sqrt{3}\gamma)=0
\]
whose first positive zero is $\gamma_{0}\approx1.97021$, yielding $\alpha_{n}\approx 8.37872\sqrt{n}$. Table~\ref{tab1} shows the corresponding values of $\lambda_{n}$ for $n$ between $3$ and $10$, together with the different possible unions of balls which also give the same $n$th eigenvalue for $\alpha=\alpha_{n}$.

The above numerical results together with Lemma~\ref{intersection} suggest that, at least in dimension two, given a fixed value of $\alpha$ there will always exist $n^*$ sufficiently large such that the minimiser for $\lambda_{n}$ is $B_n$ for all $n$ greater than $n^{*}$.

\begin{figure}[!htb]\centering
 \includegraphics[width=0.32\textwidth]{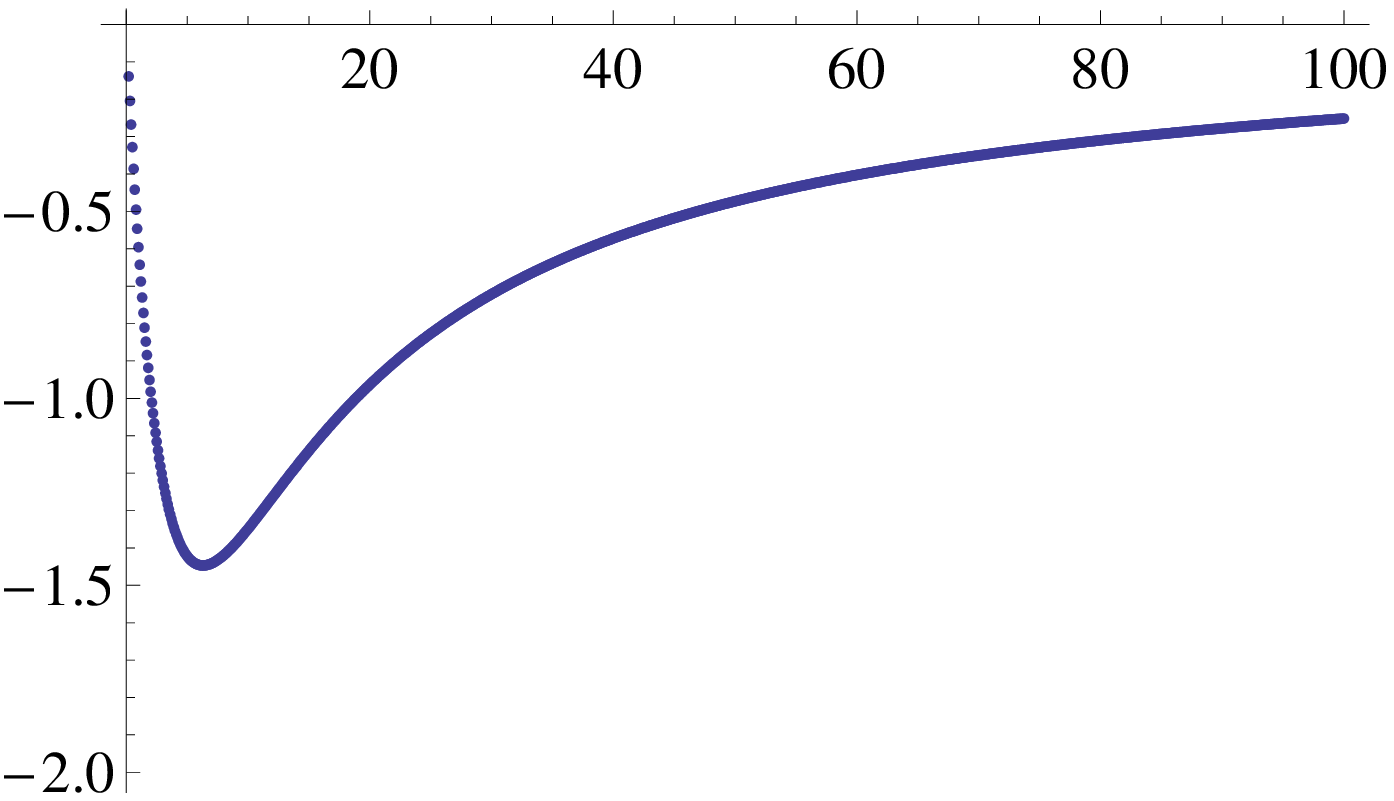} %
 \includegraphics[width=0.32\textwidth]{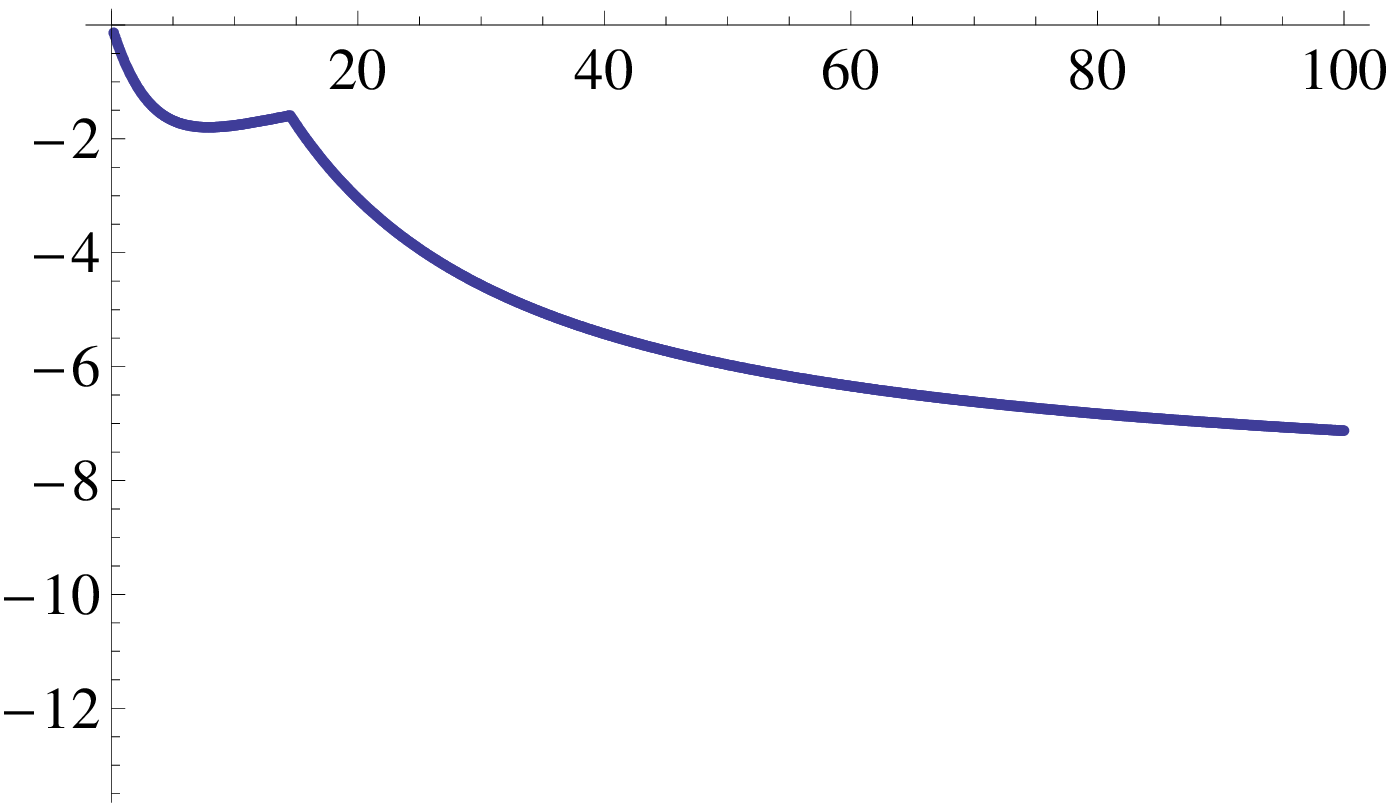} %
 \includegraphics[width=0.32\textwidth]{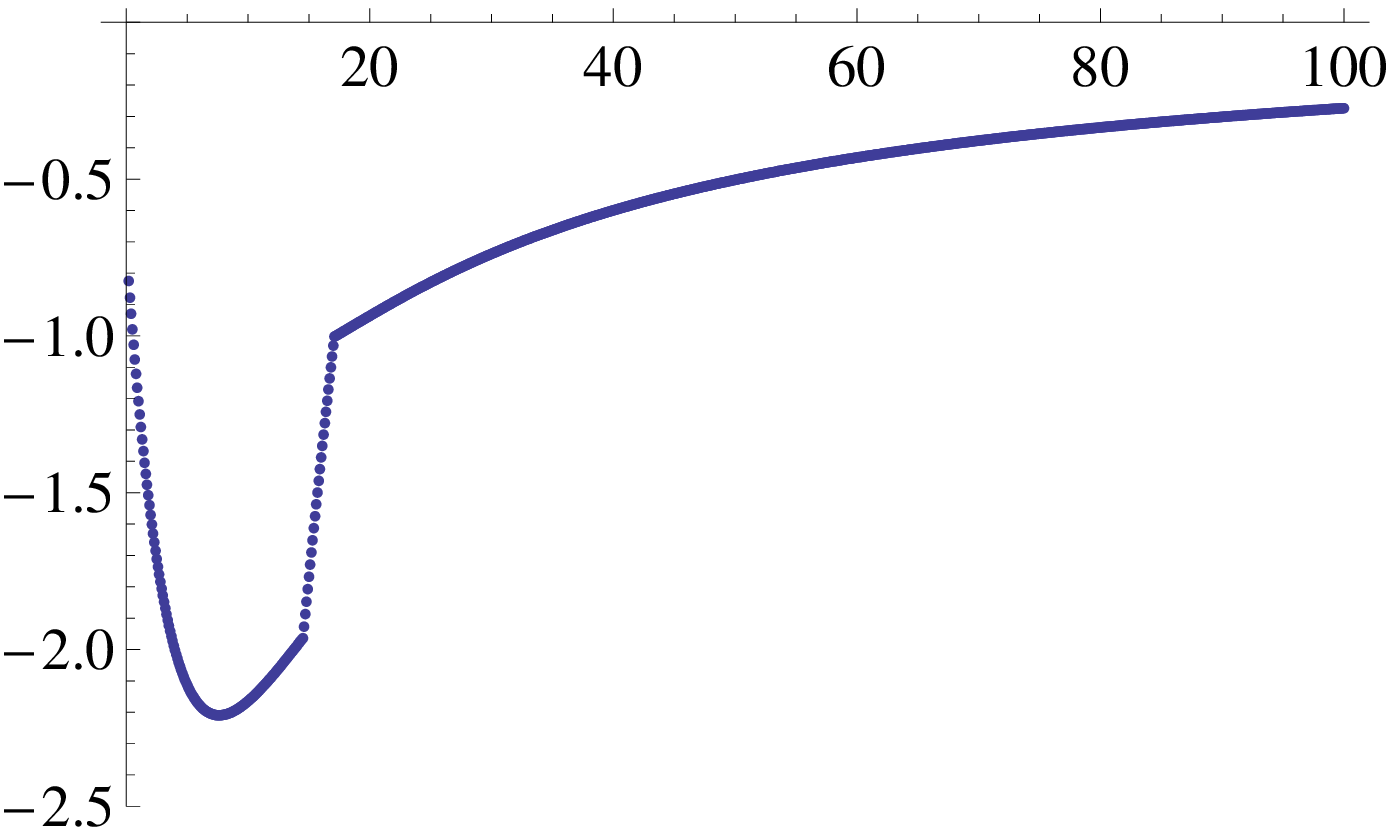} %
 \includegraphics[width=0.32\textwidth]{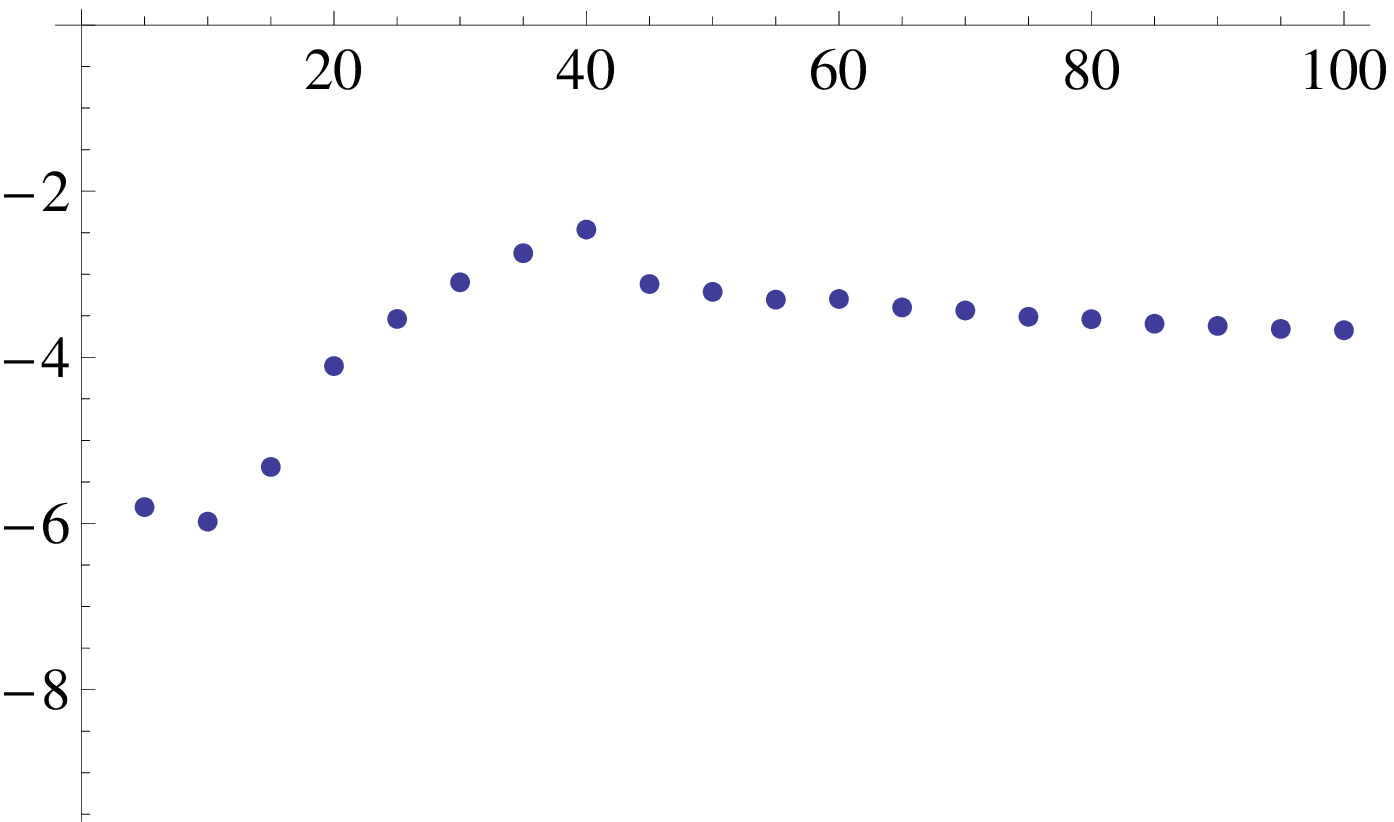} %
 \includegraphics[width=0.32\textwidth]{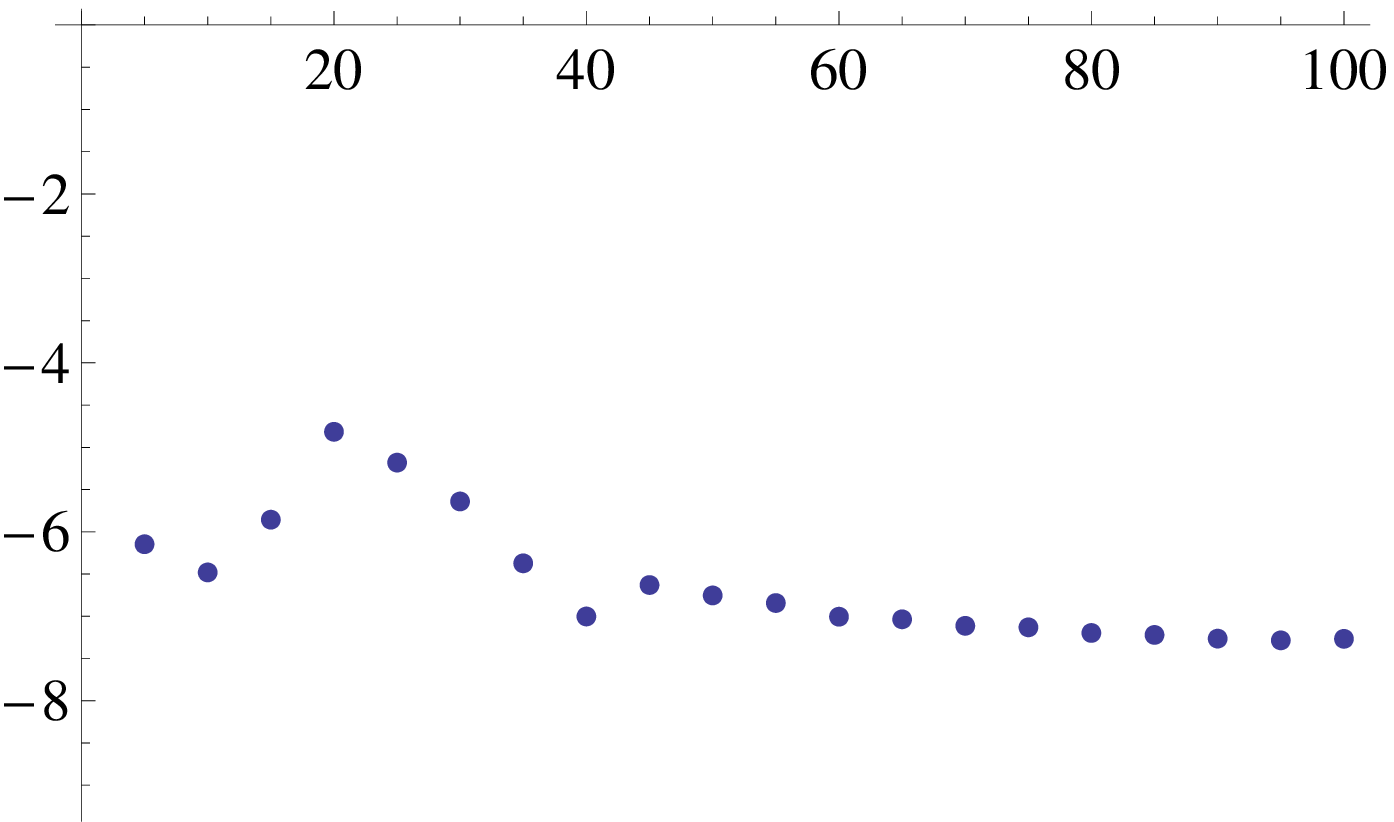} %
 \includegraphics[width=0.32\textwidth]{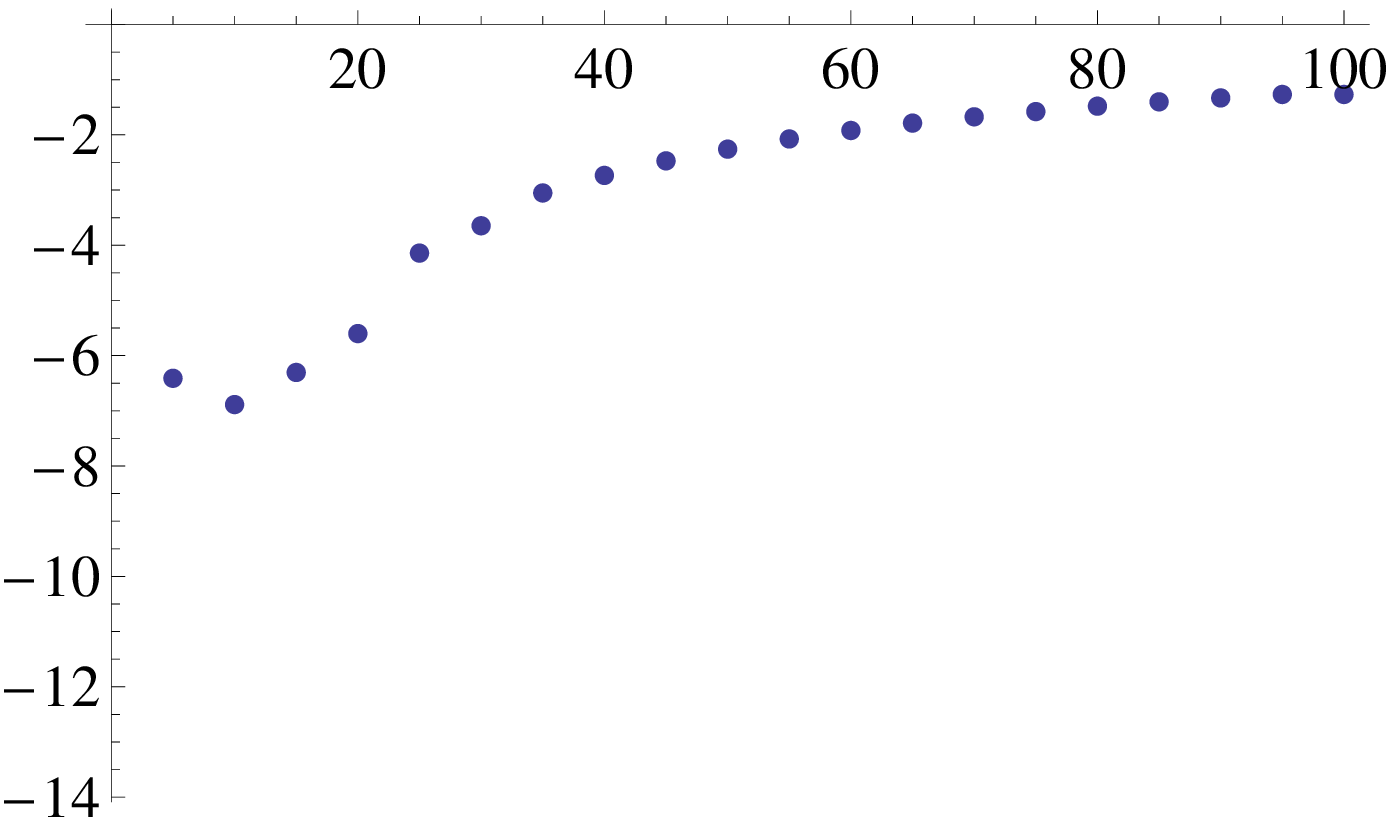} %
 \caption{Numerical values for bound \eqref{eq:comp} with n=1,2,...,6.}
 \label{fig:est}
\end{figure}

\subsection{Verification of numerical results using \eqref{eq:comp}}

As a test of the plausibility of our numerical results, we finish by computing the error in bound \eqref{eq:comp}
for the case $N=2$ and $V=1$. This is shown in Figure~\ref{fig:est} where we have plotted the quantity
\[\lambda_{n+1}^\ast-\lambda_{n}^\ast- \pi\
\lambda_1\left(B_1,\left(\frac{\left|B^\ast\right|}
{1+\left|B^\ast\right|}\right)^{\frac{1}{2}}\pi^{-\frac{1}{2}}\alpha\right)\]
as a function of $\alpha$, for n=1,2,...,6, which according to \eqref{eq:comp} must always be negative.

\section{Discussion}
By combining computational and analytical techniques we were able to address the problem of
optimizing Robin eigenvalues of the Laplacian, obtaining results for the full frequency
range and in general dimensions. The application of the MFS to Robin problems provided a fast
reliable method with which to apply a gradient--type optimization algorithm, yielding minimizers
for positive boundary parameters $\alpha$ and up to $\lambda_{7}$. From this we conclude that,
except for the first two eigenvalues, optimizers will depend on the boundary parameter and
approach the Dirichlet optimizer as $\alpha$ goes to infinity.

In order to address issues related to the connectedness of minimizers we derived a Wolf--Keller
type result for Robin problems which is also useful to identify points where there are
multiple optimizers and transitions between different branches occur. In particular,
as $\alpha$ decreases (while keeping the volume fixed), optimizers tend to become disconnected,
in contrast with the limitting Dirichlet case. From the numerical simulations and the analysis of
the transition point between the $n^{\rm th}$ eigenvalue of $n$ equal balls and that of
$n-(N+1)$ equal balls and a larger ball, we conjecture that for each $n$ there exists a transition
point, say $\alpha_{n}$, below which the optimizer for $\lambda_{n}$ consists of $n$ equal
balls and that $\alpha_{n}$ grows with $n^{1/N}$.

Finally, we were able to show that optimizers do not follow Weyl's law for the high frequencies,
growing at most with $n^{1/N}$, as opposed to the asymptotics
for a fixed domain whose leading term is of order $n^{2/N}$. As far as we are aware, it is
the first time that such behaviour has been identified.

\end{document}